\documentclass[alphabetic]{amsart}
\usepackage{enumerate, longtable}
\usepackage{amsmath, amscd, amsfonts, amsthm, amssymb, latexsym, comment, stmaryrd, graphicx, dsfont, mathtools} 
\usepackage{xcolor}
\usepackage[all]{xy}
\usepackage{fullpage}
\usepackage{tikz}
\usepackage[
        colorlinks,
        linkcolor=red,  citecolor=blue,
        backref
]{hyperref}
\usepackage{amsrefs}

\usepackage{amstext} 
\usepackage{array}
\usepackage{cleveref}

\usepackage{xcolor}

\numberwithin{equation}{section}

\newtheorem{claim}{Claim}[section]

\newtheorem{theorem}{Theorem}
\newtheorem{prop}[claim]{Proposition}
\newtheorem{cor}[claim]{Corollary}
\newtheorem{lem}[claim]{Lemma}

\theoremstyle{definition}

\newtheorem{definition}[claim]{Definition}

\newtheorem*{theorem*}{Theorem}

\newcommand{\ab}{{\mathbf a}}

\newcommand{\Z}{{\mathbb Z}}

\newcommand{\R}{{\mathbb R}}

\newcommand{\case}[1]{\vspace{0.3cm}\noindent\framebox{#1.}}

\newcommand{\Oh}[1]{O\left( {#1} \right)}
\newcommand{\parenth}[1]{\left( {#1} \right)}

\newcommand{\vol}{{\mathrm{vol}}}

\newcommand{\E}{{\mathbb{E}}}
\newcommand{\Ewp}{{\mathbb{E}_g^{WP}}}

\newcommand{\mean}[2]{\left\langle {#1} \right\rangle_{{#2}}}
\newcommand{\meanGamma}[1]{\mean{#1}{\Gamma\in L_{g,q}}}
\newcommand{\F}{\mathbb{F}}

\definecolor{dgreen}{RGB}{0,128,0}

\title{Closed geodesics in homology classes modulo sublattices}
\author{Noam Pirani}
\address{School of Mathematical Sciences, Tel Aviv University, Tel Aviv 69978, Israel.}
\email{noampirani@mail.tau.ac.il}

\begin{document}

\begin{abstract}
Let $M$ be a Weil-Petersson random hyperbolic surface of genus $g$, and let $\Gamma\subset\Z^{2g}$ be a sublattice of prime index $q$. We study the distribution of primitive closed geodesics among homology classes modulo $\Gamma$ in the large genus limit. Averaging over all such sublattices, with $q\to\infty$, we compute all the centered moments of the corresponding weighted counting functions. We exhibit a transition between Poisson and Gaussian regimes, depending on whether $\frac{X}{q\log X}$ tends to a positive constant $\lambda$ or to infinity. This ratio represents the asymptotic expected number of primitive closed geodesics in a given homology class modulo $\Gamma$. We also study the unnormalized variance $G_M(X,\Gamma)$ of the counts across homology classes. We show that, in the large genus limit, its average over sublattices of index $q$ is asymptotic to $X\log X$ as $X\to\infty$ and $q\to\infty$. Moreover, we show that the same asymptotic holds for a typical such sublattice. These results are analogous to phenomena arising in the distribution of primes in arithmetic progressions.
\end{abstract}

\maketitle

\tableofcontents

\section{Introduction}

\subsection{Prime geodesics in homology classes} Let $M$ be a closed hyperbolic surface of genus $\ge 2$. Our goal is to study the distribution of primitive closed geodesics in homology classes. For us, all geodesics are oriented, unless stated explicitly otherwise. 

Let $\gamma$ be a closed geodesic. We define its norm to be the exponent of its length
$
N(\gamma)=e^{l(\gamma)}.
$ A closed geodesic $\gamma$ is called imprimitive if $\gamma=\gamma_0^k$, $k\ge 2$ integer, for some other closed geodesic $\gamma_0$. Otherwise it is called primitive. If $\gamma=\gamma_0^k$, $k\ge 1$ integer, and $\gamma_0$ primitive, we define the von Mangoldt function $\Lambda(\gamma)=\log N(\gamma_0)$.

Define the Chebyshev theta function 
$$
\Theta_M(X)=\sum_{N(\gamma)\le X}\Lambda(\gamma).
$$
Here, and in general in this paper, a sum over geodesics is always over oriented, primitive closed geodesics, unless stated explicitly otherwise. We sometimes abbreviate and write simply $\Theta(X)$. The function $\Theta_M(X)$ is a weighted counting function of the primitive closed geodesics on $M$ of norm at most $X$. The Prime Geodesic Theorem states that
$$
\Theta_M(X)\sim X, X\to\infty.
$$

One can refine the question of counting primitive geodesics and ask instead how they are distributed among homology classes. Recall that we can identify $H_1(M,\Z)$ with the lattice $\Z^{2g}$. Let $q$ be an odd prime (a convention which we maintain throughout the paper), $\Gamma\subset \Z^{2g}$ a lattice of index $q$. Define for $\alpha\in \Z^{2g}/\Gamma$,

$$
\Theta_M(X;\Gamma,\alpha)=\sum_{N(\gamma)\le X\atop [\gamma]=\alpha\pmod{\Gamma}}\Lambda(\gamma).
$$
Here, working with primitive geodesics is important: indeed, if iterates were allowed, for any closed $\gamma$ we have $[\gamma^q]=0\pmod{\Gamma}$, and we get many trivial contributions. Phillips and Sarnak \cite{PhillipsSarnak1987} proved that for a fixed $\alpha$, as $X\to\infty$ we have
$$
\Theta_M(X;\Gamma,\alpha)\sim \frac{\Theta_M(X)}{\vol(\Gamma)}, \vol(\Gamma):=[\Z^{2g}:\Gamma].
$$
This estimate was later improved by Sharp \cite{Sharp04} to consider an estimate uniform in terms of the homology classes. 

The goal of this paper is to study the fluctuations $\Theta_M(X;\Gamma,\alpha)$ around its mean, in the large genus limit. For $k\ge 1$, define the $k$-th centered moment
$$
S_M^k(X,\Gamma)=\frac{1}{q}\sum_{\alpha\in \Z^{2g}/\Gamma}\parenth{\Theta_M(X;\Gamma,\alpha)-\frac{\Theta_M(X)}{\vol(\Gamma)}}^k.
$$
Let $L_{g,q}$ be the set of all lattices $\Gamma\subset\Z^{2g}$ of index $q$. For a function $F:L_{g,q}\to \R$ define 

$$
\meanGamma{F(\Gamma)}=\frac{1}{|L_{g,q}|}\sum_{\Gamma\in L_{g,q}}F(\Gamma).
$$
We will mostly study the average $\meanGamma{S_M^k(X,\Gamma)}$.

\subsection{Averaging over the moduli space} Let $\mathcal{M}_g$ be the moduli space of hyperbolic surfaces of genus $g$, $g\ge 2$. $\mathcal{M}_g$ is equipped with a probability measure, namely the Weil-Petersson measure. We want to compute the moments $\meanGamma{S_M^k(X,\Gamma)}$ for a random $M\in\mathcal{M}_g$, as we take the large genus limit $g\to\infty$. Explicitly, we study the asymptotics of

$$
\lim_{g\to\infty}\Ewp[\meanGamma{S_M^k(X,\Gamma}],
$$
as $X\to\infty$, and $q=q(X)$ is a sequence of odd primes which goes to $\infty$ with $X$.

We discover that there are two regimes. By the Prime Geodesic Theorem and the result of Phillips-Sarnak, heuristically we expect about $\frac{X}{q\log X}$ primitive geodesics of norm $\le X$ to fall in any fixed $\alpha\in H_1(M,\Z)/\Gamma$. In the regime $\frac{X}{q\log X}\to\infty$ we discover a Gaussian behavior, while in the regime $\frac{X}{q\log X}\to\lambda>0$ we discover a Poisson behavior.

\begin{theorem}\label{thm_gaussian_regime}
Let $k>1$ be an integer, and let $q=q(X)$ be a sequence of odd primes, with $\frac{X}{q\log X}\to\infty$, $q\to\infty$ as $X\to\infty$. Set $\sigma=\sqrt{\frac{1}{q}X\log X}$. Then,
$$
\lim_{X\to\infty}\lim_{g\to\infty}\frac{\Ewp[\meanGamma{S_M^k(X,\Gamma)}]}{\sigma^k}=\begin{cases}
    0, &k\text{ is odd},\\
    (k-1)!!, &k\text{ is even.}
\end{cases}
$$
\end{theorem}
That is, the $k$-th centered moments of $\Theta_M(X;\Gamma,\alpha)$, when averaged over all $\Gamma$ and $M\in\mathcal{M}_g$, agree with those of a standard Gaussian random variable. In the scale $\frac{X}{\log X}\asymp q$, we observe a transition to Poisson behavior. 

\begin{theorem}\label{thm_main_Poisson}
Let $k>1$, and denote $\sigma=\sqrt{\frac{1}{q}X\log X}$. Suppose that $q=q(X)$ is a sequence of odd primes with $q\to\infty$ and $\lim_{X\to\infty}\frac{X}{q\log X}=\lambda,\lambda>0$. Let $P\sim \mathrm{Poi}(\lambda)$. Then,
$$
\lim_{X\to\infty}\lim_{g\to\infty}\frac{\Ewp[\meanGamma{S_M^k(X,\Gamma)}]}{\sigma^k}=\mathbb{E}\left[\parenth{
\frac{P-\lambda}{\sqrt\lambda}}^k
\right].
$$    
\end{theorem}
We also study the unnormalized variance and its average
$$
G_M(X,\Gamma)=[\Z^{2g}:\Gamma]S_M^2(X,\Gamma)=qS_M^2(X,\Gamma), H_M(X,q)=\meanGamma{G_M(X,\Gamma)}.
$$
In a recent paper, Rudnick \cite{Rud26b} studies the quantity $G_M(X,q)$ which is essentially equal to $G_M(X, q\Z^{2g})$ (he also considers imprimitive geodesics; however, this does not change the asymptotics). Rudnick proves that 
$$
\lim_{g\to\infty}\Ewp[G_M(X,q)]\sim\begin{cases}
    X\log X, & q>2,\\
    2X\log X, & q=2.
\end{cases}
$$
The unnormalized variance $G_M(X,\Gamma)$ is closely related to a conjecture of Hooley regarding the distribution of prime numbers in arithmetic progressions, and its average value $H_M(X,q)$ is closely related to Barban-Davenport-Halberstam type theorems (see Section \ref{sec_number_theoretic_discussoin} for further discussion). 

Since the volume of $q\Z^{2g}$ is $q^{2g}\to_{g\to\infty}\infty$, we expect most homology classes mod $q$ to contain no primitive geodesics of norm $\le X$. It is therefore desirable to gain an understanding of $G_M(X,\Gamma)$ for a non-sparse limit.

We give two results which suggest that the $X\log X$ behavior persists also for non-sparse lattices. Our first result in Corollary \ref{cor_bdh_type_thm} is that as $X\to\infty$,
$$
\lim_{g\to\infty}\Ewp[H_M(X,q)]\sim \parenth{1-\frac{1}{q}}X\log X,
$$
which also holds for $q$ fixed. In particular, when we let $q\to\infty$, we get the same $ X\log X$ behavior here. Moreover, we show that the same behavior persists for a typical lattice $\Gamma$ of index $q$:

\begin{theorem}\label{thm_G_M_X_gamma_concentrations}
Let $\epsilon>0$ be fixed, and let $q=q(X)\to_{X\to\infty}\infty$ be a sequence of odd primes. We have
$$
\lim_{X\to\infty} \lim_{g\to\infty} \Ewp\left[
\frac{1}{|L_{g,q}|}\#\left\{\Gamma\in L_{g,q}: \left|\frac{G_M(X,\Gamma)}{X\log X}-1\right|>\epsilon\right\}
\right]=0.
$$
\end{theorem}
These results suggest that indeed the limit $X\log X$ is the correct behavior for the variance (rather than $X\log q$, which would be the analogous behavior if we adapt Hooley's conjecture to this context).

\textbf{Acknowledgements. } The author thanks Ze\'ev Rudnick for many helpful conversations and for comments on previous versions of this paper. This work was supported by the ISF-NSFC joint research program (Grant No. 3109/23).

\textbf{AI Declaration. } The author used OpenAI ChatGPT (GPT-6) as a conversational tool to assist with proofreading, improving the exposition, and obtaining critical feedback on the presentation and mathematical arguments. All suggestions were independently verified by the author, who takes full responsibility for the mathematical content and correctness of the results.

\section{Preliminaries}

\subsection{Mirzakhani's integration formulas}
Let $\mathcal{M}_g$ be the moduli space of all hyperbolic surfaces of given genus $g$, equipped with the Weil-Petersson measure $\operatorname{dVol}^{WP}$. We briefly discuss Mirzakhani's integration formulas, used to integrate geometric functions over $\mathcal{M}_g$. For a function $f:(0,\infty)^k\to \R$ whose support is contained in $[0,L]^k$ and a multicurve $\Gamma=(\gamma_1,\ldots,\gamma_k)$, we define a geometric  function on $\mathcal{M}_g$ via

$$
f_\Gamma(X) = \sum_{\alpha\in \operatorname{Mod}[\Gamma]}f(l_\alpha(X)).
$$
$f_\Gamma$ defines a random variable on $\mathcal{M}_g$, thought of as a probability space with the probability measure $\frac{\operatorname{dVol}^{WP}}{V_g}$. Mirzakhani's integration formula allows to compute the expected value of $f_\Gamma$, 

\begin{equation}\label{eqn_Mirzakhanis_integration_formula}
\Ewp[f_\Gamma]:=\frac{1}{V_g}\int_{\mathcal{M}_g}f_\Gamma(X)\operatorname{dVol}^{WP}(X)=\frac{C_\Gamma}{V_g}\int_{\R_+^k}f(l_1,\ldots,l_k)V_{g}(\Gamma,\mathbf l)l_1\cdots l_k dl_1\cdots dl_k.
\end{equation}
Here, $V_{g}(\Gamma,\mathbf l)$ is the volume of the moduli space of hyperbolic surfaces of genus $g$ with the multicurve $\Gamma$ removed and with $2k$ boundary components of lengths $l_1,l_1,l_2,l_2,\ldots,l_k,l_k$. 
In case the cut surface $M-\Gamma$ remains connected,  $C_\Gamma=2^{-k}$. Incorporating into \eqref{eqn_Mirzakhanis_integration_formula} the estimates for the volume ratios \cite[Proposition 3.1]{MP19}, we get that for a multicurve $\Gamma$ whose complement is connected

\begin{equation}\label{eqn_mirzakhanis_integration_with_vol_estimate}
\frac{1}{V_g}\int_{\mathcal{M}_g}f_\Gamma(X)\operatorname{dVol}^{WP}(X)=2^{-k}\int_{\R_+^k}f(l_1,\ldots,l_k)\prod_{i=1}^k \parenth{\frac{\sinh(l_i/2)}{l_i/2}}^2\parenth{1+O_{k,L}\parenth{\frac{1}{g}}}l_i dl_i.
\end{equation}

\begin{definition}
Let $\gamma_1,\ldots,\gamma_k$ be different primitive geodesics on $M$, such that the underlying unoriented geodesics are distinct too. We call them a non-SNS $k$-tuple of geodesics (or inadmissible) if they intersect themselves or each other, or if cutting $M$ along $\gamma_1,\ldots,\gamma_k$ results in a disconnected surface. Otherwise, we call them an SNS $k$-tuple.
\end{definition}

Let $k\le g$. If $\gamma_1,\ldots,\gamma_k$ are SNS then the vectors 
$$
[\gamma_1],\ldots,[\gamma_k]\in H_1(M,\Z)
$$
are linearly independent (in fact, they can be completed to a symplectic basis- see \cite[\S 1.3]{FarbMargalit}), a fact which we use often.

\begin{lem}\label{lem_non_SNS_are_negligible}
Let $f_1,\ldots,f_k:\R\to\R$ be functions bounded uniformly by a constant $C$ independent of $g$, and assume they have compact support. Then
$$
\lim_{g\to\infty}\Ewp\left[
\sum_{\gamma_1,\ldots,\gamma_k\atop{\text{inadmissible}}} f_1(l(\gamma_1))\cdots f_k(l(\gamma_k))
\right]=0.
$$\end{lem}

\begin{proof}
Let $X> 0$ and denote by $N_k'(0,\log X)$ be the number of $k$-tuples of primitive closed geodesics $\gamma_1,\ldots,\gamma_k$ on $M$ which are not SNS. Mirzakhani and Petri \cite[proof of Proposition 4.2 and Proposition 4.5]{MP19} show that
$$
\Ewp[N_k'(0,\log X)]=O_X(1/g),
$$
from which the result follows immediately. 
\end{proof}

\begin{definition}
Let $d\ge 1$ be an integer. We define
$$I_d(X)=\frac{1}{2}\int_0^{\log(X)}l^{d+1}\parenth{\frac{\sinh(l/2)}{l/2}}^2dl.
$$
\end{definition}

\begin{prop}\label{prop_J_d_principal_bound}
Let $d\ge 1$ be an integer. We have, as $X\to\infty$, 
$$
I_{d}(X)= \frac{1}{2}X\log(X)^{d-1}+\Oh{X\log(X)^{d-2}}.
$$
\end{prop}

\begin{proof} For $d=1$ we compute directly using integration by parts that 

$$
I_1(X)=\frac{1}{2}X-\frac{1}{2X}-\log X.
$$
For $d> 1$ computing explicitly, we see that 
$$
\int_0^{\log(X)}l^{d+1}\parenth{\frac{\sinh(l/2)}{l/2}}^2dl=\int_0^{\log(X)}l^{d-1}(e^l+e^{-l}-2)dl=\int_0^{\log(X)}l^{d-1}e^l+\Oh{\log(X)^d}.
$$
Using integration by parts we get 
$$
\int_0^{\log(X)}l^{d-1}e^l=[l^{d-1}e^l]^{\log(X)}_0-(d-1)\int_0^{\log(X)}l^{d-2}e^l=X\log(X)^{d-1}+\Oh{X\log(X)^{d-2}}.
$$
\end{proof}

\begin{prop}\label{prop_collision_integrals}
Let $k\ge 1, d_1,\ldots,d_k\ge 1$ be integers. We have

$$
\lim_{g\to\infty}\Ewp\left[\sum_{\gamma_1,\ldots,\gamma_k \text{ unoriented, primitive  SNS}\atop{l(\gamma_i)\le \log X}}l(\gamma_1)^{d_1}\cdots l(\gamma_k)^{d_k}\right]=\prod_{i=1}^k I_{d_i}(X).
$$
\end{prop}
\begin{proof}
This follows directly from Mirzakhani's integration formula \eqref{eqn_mirzakhanis_integration_with_vol_estimate}.
\end{proof}

\begin{definition}
Let $1\le d\le k$ be two integers. A surjective function $\pi:\{1,\ldots,k\}\to\{1,\ldots,d\}$ is called a set partition. If $g\in S_d$ is permutation, $g\circ\pi$ is another set partition, and we identify between the two, while choosing arbitrarily one of the $d!$ options to be a fixed representative. $d$ is called the length of the partition $\pi$, and the corresponding partition of $k$ is denoted $\ab(\pi)$.
\end{definition}

For $\ab\vdash k$, we denote 
$$
c(\ab)=\#\{\pi\text{ set partition of }\{1,\ldots,k\}:\ab(\pi)=\ab\}.
$$

\begin{definition}
Let $(\gamma_1,\ldots,\gamma_k)$ be a tuple of primitive closed geodesics, and let $\pi$ be a set partition of length $d$. We say that there is a collision of type $\pi$ in the tuple if there exist $d$ primitive geodesics $\eta_1,\ldots,\eta_d$, $\eta_i\neq \eta_j,\eta_j^{-1}$ for all $i\neq j$ and integers $n_1,\ldots,n_k\in\{\pm 1\}$, such that $\gamma_i=\eta_{\pi(i)}^{n_i}$. Let $\ab\vdash k$ be a partition of $k$. We say that a tuple with collision type $\pi$ has collision of type $\ab$ if $\ab(\pi)=\ab$.
\end{definition}

Our analysis will be built on this notion of collision types, each giving a different contribution in the limit $g\to\infty$.

\section{The $k$-th moment in the Gaussian regime}

\subsection{Some simplifications}

\begin{prop}\label{prop_recursion_for_S_M_k}
We have 
$$
S_M^k(X,\Gamma)=\frac{1}{q}\sum_{l(\gamma_i)\le \log X} \Lambda(\gamma_1)\cdots\Lambda(\gamma_k)1_{[\gamma_1]=\cdots=[\gamma_k]\pmod{\Gamma}}-\frac{\Theta(X)^k}{q^k}-\sum_{j=1}^{k-2} \binom{k}{j}\parenth{\frac{\Theta(X)}{q}}^jS_M^{k-j}(X,\Gamma).
$$
\end{prop}
\begin{proof}
We have 
$$
S_M^k(X,\Gamma)=\mean{\parenth{\Theta(X;\Gamma,\alpha)-\mean{\Theta(X;\Gamma,\alpha)}{\alpha\in \Z^{2g}/\Gamma})}^k}{\alpha\in\Z^{2g}/\Gamma}.
$$
We have
$$
\mean{\Theta(X;\Gamma,\alpha)^k}{\alpha\in\Z^{2g}/\Gamma}=
\mean{\parenth{\Theta(X;\Gamma,\alpha)-\mean{\Theta(X;\Gamma,\alpha)}{\alpha\in \Z^{2g}/\Gamma}+\mean{\Theta(X;\Gamma,\alpha)}{\alpha\in \Z^{2g}/\Gamma}}^k}{\alpha\in\Z^{2g}/\Gamma}.
$$
Now the result follows from Newton's binomial formula, after noting that 
$$
\mean{\Theta(X;\Gamma,\alpha)^k}{\alpha\in\Z^{2g}/\Gamma}=\frac{1}{q} \sum_{l(\gamma_i)\le \log X} \Lambda(\gamma_1)\cdots\Lambda(\gamma_k)1_{[\gamma_1]=\cdots=[\gamma_k]\pmod{\Gamma}}.
$$
\end{proof}

\begin{prop}\label{prop_k_vectors_on_lattice}
Let $g> k\ge 1$ be two integers, and let $M$ be a compact hyperbolic surface of genus $g$. Let $\gamma_1,\ldots,\gamma_k$ be closed
geodesics on $M$, and set $u_i=[\gamma_i]-[\gamma_k]\pmod{q}$, viewed as an element of $\F_q^{2g}$. We have 
$$
\meanGamma{1_{[\gamma_1]=\cdots=[\gamma_k]\pmod{\Gamma}}}=\frac{q^{2g-d}-1}{q^{2g}-1}=q^{-d}+O(q^{-2g}),
$$
where $d=\dim_{\F_q}(\mathrm{span}(u_1,\ldots,u_{k-1}))$, and the implicit constant is absolute.
\end{prop}
\begin{proof}
We have 
$$
1_{[\gamma_1]=\cdots=[\gamma_k]\pmod{\Gamma}}=\prod_{i=1}^{k-1}1_{[\gamma_i]-[\gamma_k]=0\pmod{\Gamma}}.
$$
We thus want to compute the expected value (over $\Gamma\in L_{g,q}$) of the indicator that all vectors $[\gamma_i]-[\gamma_k]\in\Gamma$. Since $\Gamma\in L_{g,q}$ corresponds to hyperplanes in $\F_q^{2g}$, we can equivalently count the number of hyperplanes in $\F_q^{2g}$ containing all the vectors $u_1,\ldots,u_{k-1}$. Such a hyperplane corresponds to a hyperplane in $\F_q^{2g}/\mathrm{span}(u_1,\ldots,u_{k-1})\cong \F_q^{2g-d}$, hence the exact result that the average is $\frac{q^{2g-d}-1}{q^{2g}-1}$. The asymptotic estimate follows from 
$$
\frac{q^{2g-d}-1}{q^{2g}-1}-q^{-d}=\frac{q^{-d}-1}{q^{2g}-1}\ll \frac{1}{q^{2g}-1}\ll q^{-2g}.
$$
\end{proof}

\begin{cor}\label{cor_k_vectors_on_lattice_in_terms_of_functionals}
Let $g> k\ge 1$ be two integers, and let $M$ be a compact hyperbolic surface of genus $g$. Let $\gamma_1,\ldots,\gamma_k$ be closed
geodesics on $M$, and set $u_i=[\gamma_i]-[\gamma_k]\pmod{q}$, viewed as an element of $\F_q^{2g}$. We have 
$$
\meanGamma{1_{[\gamma_1]=\cdots=[\gamma_k]\pmod{\Gamma}}}=\#\{(\alpha_1,\ldots,\alpha_{k-1})\in \F_q^{k-1}:\sum_{i=1}^{k-1}\alpha_iu_i=0\}\cdot q^{-(k-1)}+O(q^{-2g}).
$$
\end{cor}

\begin{proof}
We have 
$$
q^{-d}=q^{k-1-d}\cdot q^{-(k-1)},d=\dim_{\F_q}(\mathrm{span}(u_1,\ldots,u_{k-1}).
$$
Here, $q^{k-1-d}$ is the number of different functionals in $k-1$ variables vanishing on $u_1,\ldots,u_{k-1}$, therefore the result now follows from Proposition \ref{prop_k_vectors_on_lattice}.
\end{proof}

\begin{definition}
Let $(\alpha_1,\ldots,\alpha_{k-1})\in\F_q^{k-1}$. Define 
$$
C_{(\alpha_1,\ldots,\alpha_{k-1})}^k(X)=\sum_{N(\gamma_{1}),\ldots,N(\gamma_k)\le X\atop{\alpha_1u_1+\cdots+\alpha_{k-1}u_{k-1}=0}}\Lambda(\gamma_1)\cdots\Lambda(\gamma_k),
$$
where $u_i=[\gamma_i]-[\gamma_k]$.
\end{definition}

\begin{prop}\label{prop_j_coordinates_zero_recursion}
Let $j=1,\ldots,k-2$, and let $(\alpha_1,\ldots,\alpha_{k-1})\in\F_q^{k-1}$ be such that $\alpha_1=\cdots=\alpha_j=0$. Then
$$
C_{(\alpha_1,\ldots,\alpha_{k-1})}^k(X)=\Theta(X)^j C_{(\alpha_{j+1},\ldots,\alpha_{k-1})}^{k-j}(X).
$$
\end{prop}

\begin{proof}
We have
\begin{multline*}
    C_{(\alpha_1,\ldots,\alpha_{k-1})}^k(X)=\sum_{N(\gamma_{1}),\ldots,N(\gamma_k)\le X\atop{\alpha_1u_1+\cdots+\alpha_{k-1}u_{k-1}=0}}\Lambda(\gamma_1)\cdots\Lambda(\gamma_k)=\\
    =\sum_{N(\gamma_1),\ldots,N(\gamma_j)\le X}\Lambda(\gamma_1)\cdots\Lambda(\gamma_j)\sum_{N(\gamma_{j+1}),\ldots,N(\gamma_k)\le X\atop{\alpha_{j+1}u_{j+1}+\cdots+\alpha_{k-1}u_{k-1}}=0}\Lambda(\gamma_{j+1})\cdots\Lambda(\gamma_{k})=\Theta(X)^jC_{(\alpha_{j+1},\ldots,\alpha_{k-1})}^{k-j}(X).
\end{multline*}
\end{proof}

\begin{prop}\label{prop_all_coordinates_zero_recursion}
Let $k>1$ be an integer. Then,
$$
C_{(0,\ldots,0)}^k(X)=\Theta(X)^k.
$$
\end{prop}

\begin{proof}
We have
$$
C_{(0,\ldots,0)}^k(X)=\sum_{N(\gamma_1),\ldots,N(\gamma_k)\le X}\Lambda(\gamma_1)\cdots\Lambda(\gamma_k)=\Theta(X)^k.
$$
\end{proof}

\begin{prop}\label{prop_S_M_mean_as_sum_over_C_alphas}
We have for $k\ge 1$ integer
$$
\meanGamma{S_M^k(X,\Gamma)}=\frac{1}{q^k}\sum_{(\alpha_1,\ldots,\alpha_{k-1})\in \F_q^{k-1}\atop{\alpha_1\cdots\alpha_{k-1}\cdot(\alpha_1+\cdots+\alpha_{k-1})\neq 0}}C_{(\alpha_1,\ldots,\alpha_{k-1})}^k(X)+\Oh{\frac{\Theta(X)^k}{q^{2g}}}
$$
\end{prop}

\begin{proof}
We prove this by induction on $k$. The case of $k=1$ is trivial. Now, using Proposition \ref{prop_recursion_for_S_M_k} and Corollary \ref{cor_k_vectors_on_lattice_in_terms_of_functionals} we get that

$$
\meanGamma{S_M^k(X,\Gamma)}=\frac{1}{q^k}\sum_{(\alpha_1,\ldots,\alpha_{k-1})\in\F_{q}^{k-1}} C_{(\alpha_1,\ldots,\alpha_{k-1})}^k(X)-\frac{\Theta(X)^k}{q^k}-\sum_{j=1}^{k-2} \binom{k}{j}\parenth{\frac{\Theta(X)}{q}}^j\meanGamma{S_M^{k-j}(X,\Gamma)}.
$$

Let $j=1,\ldots,k-2$. Suppose that among the $k$ variables $\alpha_1,\ldots,\alpha_{k-1},\alpha_1+\cdots+\alpha_{k-1}$, precisely $j$ are zero. In this case we say that the point $(\alpha_1,\ldots,\alpha_{k-1})$ is $j$-degenerate. Then, up to relabeling of the geodesics $\gamma_1,\ldots,\gamma_k$, due to Proposition \ref{prop_j_coordinates_zero_recursion} we have
$$
\sum_{(\alpha_1,\ldots,\alpha_{k-1})\in\F_{q}^{k-1}\atop{\text{is  }j\text{-degenerate}}}C^{k}_{(\alpha_1,\ldots,\alpha_{k-1})}(X)=\binom{k}{j}\Theta(X)^j \sum_{(\alpha_{j+1},\ldots,\alpha_{k-1})\in\F_q^{k-j-1}\atop{\alpha_{j+1}\cdots\alpha_{k-1}\cdot (\alpha_{j+1}+\cdots+\alpha_{k-1})\neq 0}}C_{(\alpha_{j+1},\ldots,\alpha_{k-1})}^{k-j}(X).
$$
Moreover, using Proposition \ref{prop_all_coordinates_zero_recursion} we get that
$$
C_{(0,\ldots,0)}^k(X)=\Theta(X)^k.
$$
Combining both of these results, and using the induction hypothesis, we get the result.
\end{proof}

\subsection{Estimating $C_{(\alpha_1,\ldots,\alpha_{k-1})}^k(X)$ via collision types}

\begin{definition}
Let $(\alpha_1,\ldots,\alpha_{k-1})\in \F_q^{k-1}$, $\pi$ a set partition of $\{1,\ldots,k\}$. Define
$$
C_{(\alpha_1,\ldots,\alpha_{k-1})}^k(X,\pi)=\sum_{N(\gamma_1),\ldots,N(\gamma_k)\le X\atop{\alpha_1u_1+\cdots+\alpha_{k-1}u_{k-1}=0\atop{\gamma_1,\ldots,\gamma_k\text{ have collision type }\pi}}}\Lambda(\gamma_1)\cdots\Lambda(\gamma_k),
$$
and similarly for $\ab\vdash k$ define
$$
C_{(\alpha_1,\ldots,\alpha_{k-1})}^k(X,\ab)=\sum_{N(\gamma_1),\ldots,N(\gamma_k)\le X\atop{\alpha_1u_1+\cdots+\alpha_{k-1}u_{k-1}=0\atop{\gamma_1,\ldots,\gamma_k\text{ have collision type }\ab}}}\Lambda(\gamma_1)\cdots\Lambda(\gamma_k).
$$
\end{definition}

For the rest of the section, given $(\alpha_1,\ldots,\alpha_{k-1})\in\F_q^{k-1}$ we let $\alpha_k=-(\alpha_1+\cdots+\alpha_{k-1})$ unless explicitly stated otherwise.

\begin{definition}
Let $\pi$ be a set partition of $\{1,\ldots,k\}$ of length $d$, and let $\epsilon_1,\ldots,\epsilon_k\in\{\pm 1\}$. Define the system of $d$ linear equations in the variables $\alpha_1,\ldots,\alpha_{k-1}$ by
$$
\mathcal{S}(\pi;\epsilon_1,\ldots,\epsilon_k)=\left\{\sum_{i:\pi(i)=j} \alpha_i\epsilon_i=0,j=1,\ldots,d\right\}.
$$
Also, define 
$$\mathcal{S}(\pi)=\sum_{\epsilon_1,\ldots,\epsilon_k\in\{\pm 1\}^k}\#\{\text{Solutions }(\alpha_1,\ldots,\alpha_{k-1})\in\F_{q}^{k-1}\text{ to }\mathcal{S}(\pi;\epsilon_1,\ldots,\epsilon_k)\text{ with $\alpha_1\cdots\alpha_k\neq 0$}\}.$$ 
\end{definition}

\begin{prop}\label{prop_sum_c_alphas_pi_in_terms_of_integrals_and_S_pi}
Let $\ab\vdash k$ a partition of $k$ of length $d=l(\ab)$, and let $\pi$ be a set partition of $\{1,\ldots,k\}$, whose corresponding partition of $k$ is $\ab$. Then, 
$$
\sum_{(\alpha_1,\ldots,\alpha_{k-1})\in\F_q^{k-1}\atop{\alpha_1\cdots\alpha_{k}\neq 0}}\lim_{g\to\infty} \Ewp[C_{(\alpha_1,\ldots,\alpha_{k-1})}^k(X,\pi)]=\mathcal{S}(\pi)\prod_{j=1}^d I_{a_j}(X).
$$
\end{prop}

\begin{proof}
Using Lemma \ref{lem_non_SNS_are_negligible}, we get that

\begin{equation}\label{eq_C_alphas_ab_as_sum_over_primitive_sns}
\lim_{g\to\infty}\Ewp[C_{(\alpha_1,\ldots,\alpha_{k-1})}^k(X,\pi)]=\lim_{g\to\infty}\Ewp\left[
\sum_{N(\gamma_1),\ldots,N(\gamma_k)\le X\atop{\alpha_1 u_1+\cdots+\alpha_{k-1}u_{k-1}=0\atop{\gamma_1,\ldots,\gamma_k \text{ have collision type }\pi\atop{\text{primitive, SNS}}}}}\Lambda(\gamma_1)\cdots\Lambda(\gamma_k).
\right]
\end{equation}

Fix a $k$-tuple $\gamma_1,\ldots,\gamma_k$ with collision type $\pi$. That is, if $\pi(i)=\pi(j)=r$ then there is a simple, oriented geodesic $\eta_{r}$ such that $\gamma_i=\eta_r^{\epsilon_i}$, $\gamma_j=\eta_r^{\epsilon_j}$. Then, denoting $\alpha_k=-(\alpha_1+\cdots+\alpha_{k-1})$, the homology condition 
$$
\alpha_1u_1+\cdots+\alpha_{k-1}u_{k-1}=0
$$
translates to the system of $d$ linear equations $\mathcal{S}(\pi;\epsilon_1,\ldots,\epsilon_k)$, in the $k-1$ variables $\alpha_1,\ldots,\alpha_{k-1}$. The tuple $\gamma_1,\ldots,\gamma_k$ contributes to the sum \eqref{eq_C_alphas_ab_as_sum_over_primitive_sns} if and only if $\mathcal{S}(\pi;\epsilon_1,\ldots,\epsilon_k)$ is satisfied for $\alpha_1,\ldots,\alpha_{k-1}$. If we consider the contribution of the underlying $k$-tuple of simple unoriented geodesics $\eta_1,\ldots,\eta_d$, then they contribute to the sum \eqref{eq_C_alphas_ab_as_sum_over_primitive_sns} an amount of
$$
\sum_{\epsilon_1,\ldots,\epsilon_k\in \{\pm 1\}^k} 1_{\mathcal{S}(\pi;\epsilon_1,\ldots,\epsilon_k)\text{ is satisfied}\atop{\text{for }(\alpha_1,\ldots,\alpha_{k-1})}}\Lambda(\eta_1)^{a_1}\cdots\Lambda(\eta_d)^{a_d}.
$$

When we sum over all possible $(\alpha_1,\ldots,\alpha_{k-1})\in\F_q^{k-1}$ with $\alpha_1\cdots\alpha_{k}\neq 0$, we get that the contribution of the $k$-tuple of simple, unoriented geodesics $\eta_1,\ldots,\eta_d$ is

$$
\mathcal{S}(\pi)\Lambda(\eta_1)^{a_1}\cdots\Lambda(\eta_d)^{a_d}.
$$

We get that

$$
\sum_{(\alpha_1,\ldots,\alpha_{k-1})\in\F_q^{k-1}\atop{\alpha_1\cdots\alpha_{k}\neq 0}}\lim_{g\to\infty} \Ewp[C_{(\alpha_1,\ldots,\alpha_{k-1})}^k(X,\pi)]=\mathcal{S}(\pi)\lim_{g\to\infty}\Ewp\left[
\sum_{N(\eta_1),\ldots,N(\eta_d)\le X\atop{\text{unoriented, SNS}}}\Lambda(\eta_1)^{a_1}\cdots\Lambda(\eta_d)^{a_d}
\right].
$$

Using Propositions \ref{prop_collision_integrals} and \ref{prop_J_d_principal_bound}, we see that

$$
\sum_{(\alpha_1,\ldots,\alpha_{k-1})\in\F_q^{k-1}\atop{\alpha_1\cdots\alpha_{k}\neq 0}}\lim_{g\to\infty}\Ewp[C_{(\alpha_1,\ldots,\alpha_{k-1})}^k(X,\pi)]=\mathcal{S}(\pi)\prod_{j=1}^d I_{a_j}(X).
$$
\end{proof}

\begin{cor}\label{cor_partitions_with_block_of_1_do_not_contribute} Let $\ab\vdash k$ be a partition of $k$ with at least one part equal to $1$. Then, we have
$$
\sum_{(\alpha_1,\ldots,\alpha_{k-1})\in\F_q^{k-1}\atop{\alpha_1,\ldots\alpha_k\neq 0}}\lim_{g\to\infty}\Ewp[C_{(\alpha_1,\ldots,\alpha_{k-1})}^k(X,\ab)]=0.
$$
\end{cor}

\begin{proof} We have

$$
C_{(\alpha_1,\ldots,\alpha_{k-1})}^k(X,\ab)=\sum_{\pi\text{ set partition of}\atop{\{1,\ldots,k\}\atop{\text{with }\ab(\pi)=\ab}}} C_{(\alpha_1,\ldots,\alpha_{k-1})}^k(X,\pi).
$$
Therefore, using Proposition \ref{prop_sum_c_alphas_pi_in_terms_of_integrals_and_S_pi}, we see that to prove the corollary it is enough to show that for any set partition $\pi$ with $\ab(\pi)=\ab$, we have $\mathcal{S}(\pi)=0$. Fix $\epsilon_1,\ldots,\epsilon_k\in\{\pm 1\}$. Without loss of generality, assume that $|\pi^{-1}(\pi(1))|=1$ (i.e. $1$ is in a singleton block of $\pi$). Then the system of equations $\mathcal{S}(\pi;\epsilon_1,\ldots,\epsilon_k)$ contains the equation $\epsilon_1\alpha_1=0$, which implies $\alpha_1=0$. Therefore, it has no solutions with $\alpha_1\cdots\alpha_k\neq 0$, and hence $\mathcal{S}(\pi)=0$.
\end{proof}

\begin{prop}\label{prop_bound_on_C_alphas_ab}
Let $\ab\vdash k$ be a partition of length $d=l(\ab)$. Assume that every part of $\ab$ is at least $2$. Then, as $X\to\infty$,
$$
\sum_{(\alpha_1,\ldots,\alpha_{k-1})\in\F_q^{k-1}\atop{\alpha_1\cdots\alpha_k\neq 0}}\lim_{g\to\infty}\Ewp[C_{(\alpha_1,\ldots,\alpha_{k-1})}^k(X,\ab)]\ll q^{k-d}X^{d}\log(X)^{k-d}.
$$
Here, the implied constant may depend on $k$.
\end{prop}

\begin{proof}
Let $\pi$ be a set partition of $\{1,\ldots,k\}$ with $\ab(\pi)=\ab$. We claim that the number of solutions $(\alpha_1,\ldots,\alpha_{k-1})\in\F_q^{k-1}$ to the system of equations $\mathcal{S}(\pi;\epsilon_1,\ldots,\epsilon_k)$ with $\alpha_1\cdots\alpha_k\neq 0$ is $\le q^{k-d}$. Assume without loss of generality that $\pi(k)=1$. $\mathcal{S}(\pi;\epsilon_1,\ldots,\epsilon_k)$ contains the equations
$$
\sum_{r:\pi(r)=i}\epsilon_r\alpha_r=0,2\le i\le d.
$$
These are $d-1$ linearly independent equations, since the variables involved in the $i$-th equation are $\pi^{-1}(i)$ and these sets are disjoint for $i\neq j$ (when we exclude the first equation). We get that the number of solutions to $\mathcal{S}(\pi;\epsilon_1,\ldots,\epsilon_k)\le q^{k-1-(d-1)}=q^{k-d}$. Therefore, $\mathcal{S}(\pi)\le 2^kq^{k-d}$. 

Now, using Proposition \ref{prop_sum_c_alphas_pi_in_terms_of_integrals_and_S_pi}, we get that
$$
\sum_{(\alpha_1,\ldots,\alpha_{k-1})\in\F_q^{k-1}\atop{\alpha_1\cdots\alpha_k\neq 0}}\lim_{g\to\infty}\Ewp[C_{(\alpha_1,\ldots,\alpha_{k-1})}^k(X,\pi)]\le 2^kq^{k-d}\prod_{j=1}^d I_{a_j}(X).
$$
Summing this over all $\pi$ with $\ab(\pi)=\ab$, we get (using Proposition \ref{prop_J_d_principal_bound})
$$
\sum_{(\alpha_1,\ldots,\alpha_{k-1})\in\F_q^{k-1}\atop{\alpha_1\cdots\alpha_k\neq 0}}\lim_{g\to\infty}\Ewp[C_{(\alpha_1,\ldots,\alpha_{k-1})}^k(X,\ab)]\ll q^{k-d}\prod_{j=1}^d I_{a_j}(X)\ll q^{k-d}X^d\log(X)^{k-d}.
$$
\end{proof}

\begin{cor}\label{cor_non_pairing_partitions_have_no_contribution}
Let $q=q(X)$ be a sequence of odd primes, $q\to\infty$, such that $\frac{X}{q\log X}\to\infty$. Let $\ab\vdash k$ be a partition of length $d=l(\ab)$, and denote $\sigma=\sqrt{\frac{1}{q}X\log X}$. Suppose that $d<k/2$. Then, as $X\to\infty$,
$$
\frac{1}{q^k}\sum_{(\alpha_1,\ldots,\alpha_{k-1})\in\F_q^{k-1}\atop{\alpha_1\cdots\alpha_k\neq 0}}\lim_{g\to\infty}\Ewp[C^k_{(\alpha_1,\ldots,\alpha_{k-1})}(X,\ab)]\ll\parenth{\frac{X}{q\log(X)}}^{-1/2}\sigma^k.
$$
\end{cor}

\begin{proof}
Using Corollary \ref{cor_partitions_with_block_of_1_do_not_contribute} and Proposition \ref{prop_bound_on_C_alphas_ab}, we get that
$$
\frac{1}{q^k}\sum_{(\alpha_1,\ldots,\alpha_{k-1})\in\F_q^{k-1}\atop{\alpha_1\cdots\alpha_k\neq 0}}\lim_{g\to\infty}\Ewp[C^k_{(\alpha_1,\ldots,\alpha_{k-1})}(X,\ab)]\ll q^{-d}X^d\log(X)^{k-d}.
$$
Dividing by $\sigma^k$ we get 
$$
\frac{\frac{1}{q^k}\sum_{(\alpha_1,\ldots,\alpha_{k-1})\in\F_q^{k-1}\atop{\alpha_1\cdots\alpha_k\neq 0}}\lim_{g\to\infty}\Ewp[C^k_{(\alpha_1,\ldots,\alpha_{k-1})}(X,\ab)]}{\sigma^k}\ll q^{k/2-d}X^{d-k/2}\log(X)^{k/2-d}=\parenth{\frac{X}{q\log(X)}}^{d-k/2}.
$$
We finish the proof by noticing that since $k,d$ are positive integers, with $d<k/2$, $k/2-d\ge 1/2$.
\end{proof}

\begin{cor}\label{cor_main_term_contribution_is_only_of_pairings}
Let $q=q(X)$ be a sequence of odd primes, $q\to\infty$, such that $\frac{X}{q\log X}\to\infty$. Denote $\sigma=\sqrt{\frac{1}{q}X\log X}$. As $X\to\infty$ we have
\begin{multline*}
\frac{1}{q^k}\sum_{(\alpha_1,\ldots,\alpha_{k-1})\in\F_q^{k-1}\atop{\alpha_1\cdots\alpha_k\neq 0}}\lim_{g\to\infty}\Ewp[C_{(\alpha,\ldots,\alpha_{k-1})}^k(X)]=\\
=\frac{1}{q^k}\sum_{\pi\text{ set partition of }\atop{\{1,\ldots,k\}\atop{\ab(\pi)=(2,\ldots,2)}}}\sum_{(\alpha_1,\ldots,\alpha_{k-1})\in\F_q^{k-1}\atop{\alpha_1\cdots\alpha_k\neq 0}}\lim_{g\to\infty}\Ewp [C_{(\alpha_1,\ldots,\alpha_{k-1})}^k(X,\pi)]
+\Oh{\sigma^k\parenth{\frac{X}{q\log X}}^{-1/2}}.    
\end{multline*}
In particular if $k$ is odd,
$$
\frac{1}{q^k}\sum_{(\alpha_1,\ldots,\alpha_{k-1})\in\F_q^{k-1}\atop{\alpha_1\cdots\alpha_k\neq 0}}\lim_{g\to\infty}\Ewp[C_{(\alpha,\ldots,\alpha_{k-1})}^k(X)]=\Oh{\sigma^k\parenth{\frac{X}{q\log X}}^{-1/2}}.
$$
\end{cor}
\begin{proof}
This is a direct consequence of corollaries \ref{cor_partitions_with_block_of_1_do_not_contribute} and \ref{cor_non_pairing_partitions_have_no_contribution}.
\end{proof}

\subsection{The contribution of pairings and the $k$-th moment}

To compute $\lim_{g\to\infty}\Ewp[\meanGamma{S_M^k(X,\Gamma)}]$, using Corollary \ref{cor_main_term_contribution_is_only_of_pairings}, the remaining ingredient is to understand the contribution of pairings $\pi$- set partitions of $\{1,\ldots,k\}$ with all parts equal to $2$. We do this by estimating the quantity $\mathcal{S}(\pi)$.

\begin{prop}\label{prop_S_pi_of_pairings}
Let $k=2d$ be an even integer, $\pi$ be a pairing of $\{1,\ldots,k\}$. Then, 

$$\mathcal{S}(\pi)=2^dq^{k-d}+\Oh{q^{k-d-1}}.$$
\end{prop}

\begin{proof}
Call $\epsilon_1,\ldots,\epsilon_k\in\{\pm 1\}^k$ $\pi$-compatible if for all $1\le i,j\le k$, if $\pi(i)=\pi(j)$ then $\epsilon_i=\epsilon_j$. We claim that:

\begin{enumerate}
    \item If $\epsilon_1,\ldots,\epsilon_k$ are $\pi$-compatible, the number of solutions $\alpha_1,\ldots,\alpha_{k-1}\in\F_q^{k-1}$ of $S(\pi;\epsilon_1,\ldots,\epsilon_k)$ with $\alpha_1\cdots\alpha_k\neq 0$ is $q^{k-d}+O(q^{k-d-1})$. 
    \item If $\epsilon_1,\ldots,\epsilon_k$ are not $\pi$-compatible, the number of solutions  $\alpha_1,\ldots,\alpha_{k-1}\in\F_q^{k-1}$ of $S(\pi;\epsilon_1,\ldots,\epsilon_k)$ with $\alpha_1\cdots\alpha_k\neq 0$ is $O(q^{k-d-1})$. 
\end{enumerate}
This will give the result, because the number of $\epsilon_1,\ldots,\epsilon_k$ which are $\pi$-compatible is $2^d$. 

\case{$\epsilon_1,\ldots,\epsilon_k$ are $\pi$-compatible} For $i=1,\ldots,d$ denote $1\le i_1,i_2\le k$ the two indices with $\pi(i_1)=\pi(i_2)=i$. Our system of equations is 
$$
\xi_i (\alpha_{i_1}+\alpha_{i_2})=0, i=1,\ldots,d,\xi_i\in\{\pm 1\},
$$
which is equivalent to the system of equations 
$$
\alpha_{i_1}+\alpha_{i_2}=0, i=1,\ldots,d.
$$
If we view this as a system of linear equations in the variables $\alpha_1,\ldots,\alpha_k\in\F_q^k$ with the additional equation $\alpha_1+\cdots+\alpha_k=0$, then it is a system of $d$ linearly independent equations; therefore the number of solutions is $q^{k-d}$. We need to remove solutions with either $\alpha_1,\ldots,\alpha_k=0$. Each such condition cuts a hyperplane so the number of solutions we remove is at most $kq^{k-d-1}\ll q^{k-d-1}$.

\case{$\epsilon_1,\ldots,\epsilon_k$ are not $\pi$-compatible} Using the same notation as in the previous case, our system of equations in this case is 
$$
\epsilon_{i_1}\alpha_{i_1}+\epsilon_{i_2}\alpha_{i_2}=0,i=1,\ldots,d.
$$
For at least one value of $i$, $\epsilon_{i_1}=-\epsilon_{i_2}$. Without loss of generality we assume that $\epsilon_{i_1}=1,\epsilon_{i_2}=-1$ (otherwise we multiply by a constant to get an equivalent equation). We count solutions without the restrictions that $\alpha_1,\ldots,\alpha_k\neq 0$. Recalling that $\alpha_1+\cdots+\alpha_k=0$, we view our system of equations as a system of $d+1$ linear equations in $k$ variables $\alpha_1,\ldots,\alpha_k$ over $\F_q$, namely
$$
\alpha_1+\cdots+\alpha_k=0,\alpha_{i_1}=\alpha_{i_2},\{\epsilon_{j_1}\alpha_{j_1}+\epsilon_{j_2}\alpha_{j_2=0}\}_{j=1,\ldots,d\atop j\neq i}.
$$
We claim that this system forms a system of $d+1$ linearly independent equations. Since the variables in the last $d$ equations are disjoint, they are clearly linearly independent. So to assume the contrary means that there are some $c_1,\ldots,c_d\in\F_q$ such that 
$$
\alpha_1+\cdots+\alpha_k=c_i\alpha_{i_1}-c_i\alpha_{i_2}+\sum_{j=1\atop{j\neq i}}^d(c_j\epsilon_{j_1}\alpha_{j_1}+c_j\epsilon_{j_2}\alpha_{j_2}).
$$
But this is a contradiction since comparing the coefficients of $\alpha_{i_1}$ on both sides we get $c_i=1$, and comparing the coefficients of $\alpha_{i_2}$ on both side we get $c_i=-1$. Therefore the number of solutions is at most $q^{k-d-1}$.
\end{proof}

\begin{cor}\label{cor_total_contribution_of_pairings}
Let $k=2d$ be an even integer, $\pi$ be a pairing of $\{1,\ldots,k\}$. Let $\sigma=\sqrt{\frac{1}{q}X\log X}.$ We have, as $X\to\infty$,
$$
\frac{1}{q^k}\sum_{(\alpha_1,\ldots,\alpha_{k-1})\in\F_q^{k-1}\atop{\alpha_1\cdots\alpha_k\neq0}} \lim_{g\to\infty} \Ewp[C_{(\alpha_1,\ldots,\alpha_{k-1})}(X,\pi)]=\sigma^k+\Oh{\parenth{\frac{1}{\log X}+\frac{1}{q}}\sigma^k}.
$$
\end{cor}

\begin{proof}
By Proposition \ref{prop_sum_c_alphas_pi_in_terms_of_integrals_and_S_pi}, we have
$$
\frac{1}{q^k}\sum_{(\alpha_1,\ldots,\alpha_{k-1})\in\F_q^{k-1}\atop{\alpha_1\cdots\alpha_k\neq0}} \lim_{g\to\infty} \Ewp[C_{(\alpha_1,\ldots,\alpha_{k-1})}(X,\pi)]=q^{-k}\mathcal{S}(\pi)I_2(X)^d.
$$
Using Proposition \ref{prop_S_pi_of_pairings} we get that
$$
q^{-k}\mathcal{S}(\pi)I_2(X)^d=2^dq^{-d}I_2(X)^d+\Oh{q^{-d-1}I_2(X)^d}.
$$
By Proposition \ref{prop_J_d_principal_bound}, $I_2(X)= \frac{1}{2}X\log X+\Oh{X}$. Therefore, $I_2(X)^d= 2^{-d}X^{d}\log(X)^d+\Oh{X^d\log(X)^{d-1}}$. We get that
$$
q^{-k}\mathcal{S}(\pi)I_2(X)^d=\sigma^k+\Oh{\parenth{\frac{1}{\log X}+\frac{1}{q}}\sigma^k},
$$
and the result follows.
\end{proof}

\begin{proof}[Proof of Theorem \ref{thm_gaussian_regime}]
Using Proposition \ref{prop_S_M_mean_as_sum_over_C_alphas} and Corollary \ref{cor_main_term_contribution_is_only_of_pairings} we get
\begin{multline*}
\lim_{g\to\infty}\frac{\Ewp[\meanGamma{S_M^k(X,\Gamma)}]}{\sigma^k}=\\
=\frac{1}{\sigma^kq^k}\sum_{\pi\text{ set partition of }\atop{\{1,\ldots,k\}\atop{\ab(\pi)=(2,\ldots,2)}}}\sum_{(\alpha_1,\ldots,\alpha_{k-1})\in\F_q^{k-1}\atop{\alpha_1\cdots\alpha_k\neq 0}}\lim_{g\to\infty}\Ewp[C_{(\alpha_1,\ldots,\alpha_{k-1})}^k(X,\pi)]
+\Oh{\parenth{\frac{X}{q\log X}}^{-1/2}}. 
\end{multline*}
Now using Corollary \ref{cor_total_contribution_of_pairings}, we get that
$$
\lim_{g\to\infty}\frac{\Ewp[\meanGamma{S_M^k(X,\Gamma)}]}{\sigma^k}=\#\{\text{pairings on }\{1,\ldots,k\}\}+\Oh{\frac{1}{\log X}+\frac{1}{q}+\parenth{\frac{X}{q\log X}}^{-1/2}}.
$$
Taking $X\to\infty$ we get the result.
\end{proof}

\section{The $k$-th moment in the Poisson regime}

In the Poisson regime, where $\frac{X}{q}\sim\lambda\log X$, the estimate of the error in Corollary \ref{cor_non_pairing_partitions_have_no_contribution} is not enough. And indeed, every set partition $\pi$ without a singleton block contributes to the main term of $\lim_{g\to\infty}\Ewp[\meanGamma{S_M^k(X,\Gamma)}]$. To compute the contribution we need a better estimate of $\mathcal{S}(\pi)$, similar in spirit to the estimate from Proposition \ref{prop_S_pi_of_pairings} for the case of $\pi$ a pairing.

\begin{prop}\label{prop_S_pi_general_partition_with_no_singleton_blocks}
Let $k$ be an integer, $\pi$ a set partition of $\{1,\ldots,k\}$ of length $d$ without a singleton block. Then,
$$
\mathcal{S}(\pi)=2^dq^{k-d}+\Oh{q^{k-d-1}}.
$$
\end{prop}

\begin{proof} As in the proof of Proposition \ref{prop_S_pi_of_pairings}, we call $\epsilon_1,\ldots,\epsilon_k\in\{\pm 1\}^k$ $\pi$-compatible if for all $1\le i,j\le k$, if $\pi(i)=\pi(j)$ then $\epsilon_i=\epsilon_j$. We claim that

\begin{enumerate}
    \item If $\epsilon_1,\ldots,\epsilon_k$ are $\pi$-compatible, the number of solutions $\alpha_1,\ldots,\alpha_{k-1}\in\F_q^{k-1}$ of $S(\pi;\epsilon_1,\ldots,\epsilon_k)$ with $\alpha_1\cdots\alpha_k\neq 0$ is $q^{k-d}+O(q^{k-d-1})$. 
    \item If $\epsilon_1,\ldots,\epsilon_k$ are not $\pi$-compatible, the number of solutions  $\alpha_1,\ldots,\alpha_{k-1}\in\F_q^{k-1}$ of $S(\pi;\epsilon_1,\ldots,\epsilon_k)$ with $\alpha_1\cdots\alpha_k\neq 0$ is $O(q^{k-d-1})$. 
\end{enumerate}
This would finish the claim because there are precisely $2^d$ $\pi$-compatible $\epsilon_1,\ldots\epsilon_k$ (choose a sign $\pm 1$ for each of the $d$ blocks). 

\case{$\epsilon_1,\ldots,\epsilon_k$ are $\pi$-compatible} Multiplying by an appropriate constant, the system of equations $\mathcal{S}(\pi;\epsilon_1,\ldots\epsilon_k)$ is equivalent to the system of equations 
$$
\sum_{j\in\pi^{-1}(i)}\alpha_j=0,i=1,\ldots,d.
$$
To count solutions $\alpha_1,\ldots,\alpha_{k-1}$ with $\alpha_1\cdots\alpha_k\neq 0$, $\alpha_k=-(\alpha_1+\cdots+\alpha_{k-1})$, we view this as a system of $d+1$ linear equations in the variables $\alpha_1,\ldots,\alpha_k$, namely the system
$$
\alpha_1+\cdots+\alpha_k=0,\sum_{j\in\pi^{-1}(i)}\alpha_j=0,i=1,\ldots,d.
$$
Notice that the first equation is clearly linearly dependent on the other $d$ equations, and that the rest of the equations are linearly independent; therefore we get that there are precisely $q^{k-d}$ solutions $\alpha_1,\ldots,\alpha_k\in\F_q^k$. The conditions $\alpha_1\cdots\alpha_k\neq 0$ translates to excluding either of the hyperplanes $\alpha_i=0$. Since there are no singleton blocks, none of these hyperplanes contains our solution set. Therefore it cuts out a set of solutions of size $q^{k-d-1}$ which we need to remove. In particular, the number of solutions is $q^{k-d}+\Oh{q^{k-d-1}}$.

\case{$\epsilon_1,\ldots,\epsilon_k$ are not $\pi$-compatible} In this case, we know that for some index $1\le i_0\le d$, not all of the signs $\{\epsilon_j\}_{\pi(j)=i_0}$ are equal. We want to bound the number of solutions to $\mathcal{S}(\pi;\epsilon_1,\ldots,\epsilon_k)$ without the extra restriction that $\alpha_1\cdots\alpha_k\neq 0$. We think of our system of equations as the system of linear equations in $k$ variables $\alpha_1,\ldots,\alpha_k\in\F_q^k$, 
$$
\alpha_1+\cdots+\alpha_k=0,\left\{\sum_{j\in\pi^{-1}(i)}\epsilon_j\alpha_j=0\right\}_{i=1,\ldots,d}.
$$
We claim that these $d+1$ equations are linearly independent. Clearly the last $d$ equations are linearly independent. So, if they are linearly dependent, we have $c_1,\ldots,c_d\in\F_q$ such that 
$$
\alpha_1+\cdots+\alpha_k=\sum_{i=1}^d c_i\sum_{j\in\pi^{-1}(i)}\epsilon_j\alpha_j.
$$
Comparing coefficients of $\alpha_j$ for $j\in\pi^{-1}(i_0)$, we get 
$
c_{i_0}\epsilon_j=1
$ for every $j\in\pi^{-1}(i_0)$. This is a contradiction to the assumption that not all of the signs $\{\epsilon_j\}_{j\in\pi^{-1}(i_0)}$ are equal. We get that the number of solutions of $\mathcal{S}(\pi;\epsilon_1,\ldots,\epsilon_k)$ is at most $q^{k-d-1}$.

\end{proof}

\begin{prop}\label{prop_partition_contribution_poisson_regime}
Let $\ab\vdash k$ be a partition of length $d=l(\ab)$ without singleton blocks, and denote $\sigma=\sqrt{\frac{1}{q}X\log X}$. Then, as $X\to\infty$,
$$
\frac{1}{q^k}\sum_{(\alpha_1,\ldots,\alpha_{k-1})\in\F_q^{k-1}\atop{\alpha_1\cdots\alpha_k\neq 0}}\lim_{g\to\infty}\Ewp[C^k_{(\alpha_1,\ldots,\alpha_{k-1})}(X,\ab)]=c(\ab)q^{-d}X^d\log(X)^{k-d}\parenth{1+\Oh{\frac{1}{q}+\frac{1}{\log X}}}.
$$
\end{prop}

\begin{proof}
We have
\begin{multline*}
\frac{1}{q^k}\sum_{(\alpha_1,\ldots,\alpha_{k-1})\in\F_q^{k-1}\atop{\alpha_1\cdots\alpha_k\neq 0}}\lim_{g\to\infty}\Ewp[C^k_{(\alpha_1,\ldots,\alpha_{k-1})}(X,\ab)]=\\
=\sum_{\pi\text{ set partition of }\atop{\{1,\ldots,k\}\atop{\ab(\pi)=\ab}}}\frac{1}{q^k}\sum_{(\alpha_1,\ldots,\alpha_{k-1})\in\F_q^{k-1}\atop{\alpha_1\cdots\alpha_k\neq 0}}\lim_{g\to\infty}\Ewp[C^k_{(\alpha_1,\ldots,\alpha_{k-1})}(X,\pi)].
\end{multline*}

Using Proposition \ref{prop_sum_c_alphas_pi_in_terms_of_integrals_and_S_pi}, for every $\pi$ with $\ab(\pi)=\ab$ we have
$$
\frac{1}{q^k}\sum_{(\alpha_1,\ldots,\alpha_{k-1})\in\F_q^{k-1}\atop{\alpha_1\cdots\alpha_k\neq 0}}\lim_{g\to\infty}\Ewp[C^k_{(\alpha_1,\ldots,\alpha_{k-1})}(X,\pi)]=q^{-k}\mathcal{S}(\pi)\prod_{j=1}^dI_{a_j}(X).
$$
Using propositions \ref{prop_S_pi_general_partition_with_no_singleton_blocks}, \ref{prop_J_d_principal_bound} we get that

$$
q^{-k}\mathcal{S}(\pi)\prod_{j=1}^dI_{a_j}(X)=q^{-d}X^d\log(X)^{k-d}\parenth{1+\Oh{\frac{1}{q}+\frac{1}{\log X}}}.
$$
There are precisely $c(\ab)$ set partitions $\pi$ with $\ab(\pi)=\ab$, therefore we get the result.
\end{proof}

\begin{cor}\label{cor_partition_contribution_poisson_regime}
Let $\ab\vdash k$ be a partition of length $d=l(\ab)$ without singleton blocks, and denote $\sigma=\sqrt{\frac{1}{q}X\log X}$. Suppose that $\lim_{X\to\infty}\frac{X}{q\log X}=\lambda$, $\lambda>0$. Then, 
$$
\lim_{X\to\infty}\frac{1}{\sigma^kq^k}\sum_{(\alpha_1,\ldots,\alpha_{k-1})\in\F_q^{k-1}\atop{\alpha_1\cdots\alpha_k\neq 0}}\lim_{g\to\infty}\Ewp[C^k_{(\alpha_1,\ldots,\alpha_{k-1})}(X,\ab)]=c(\ab)\lambda^{^{d-k/2}}.
$$
\end{cor}

\begin{proof}
Using Proposition \ref{prop_partition_contribution_poisson_regime}, we have

\begin{multline*}
\lim_{X\to\infty}\frac{1}{\sigma^kq^k}\sum_{(\alpha_1,\ldots,\alpha_{k-1})\in\F_q^{k-1}\atop{\alpha_1\cdots\alpha_k\neq 0}}\lim_{g\to\infty}\Ewp[C^k_{(\alpha_1,\ldots,\alpha_{k-1})}(X,\ab)]=\\
=\lim_{X\to\infty}c(\ab)\sigma^{-k}q^{-d}X^d\log(X)^{k-d}\parenth{1+\Oh{\frac{1}{q}+\frac{1}{\log X}}}.    
\end{multline*}
But
$$
\sigma^{-k}q^{-d}X^d\log(X)^{k-d}=q^{k/2-d}X^{d-k/2}\log(X)^{k/2-d}=\parenth{\frac{X}{q\log X}}^{d-k/2}.
$$
Taking $X\to\infty$, we get the result.
\end{proof}

\begin{proof}[Proof of Theorem \ref{thm_main_Poisson}]
Using Proposition \ref{prop_S_M_mean_as_sum_over_C_alphas} and Corollary \ref{cor_partitions_with_block_of_1_do_not_contribute} we get

\begin{multline*}
\lim_{X\to\infty}\lim_{g\to\infty}\frac{\Ewp[\meanGamma{S_M^k(X,\Gamma)}]}{\sigma^k}=\\
=\sum_{\ab\vdash k\atop{\ab\text{ has no singleton blocks}}}\lim_{X\to\infty}\frac{1}{\sigma^kq^k}\sum_{(\alpha_1,\ldots,\alpha_{k-1})\in\F_q^{k-1}\atop{\alpha_1\cdots\alpha_k\neq 0}}\lim_{g\to\infty}\Ewp[C^k_{(\alpha_1,\ldots,\alpha_{k-1})}(X,\ab)].
\end{multline*}
By Corollary \ref{cor_partition_contribution_poisson_regime} we get that 
$$
\lim_{X\to\infty}\lim_{g\to\infty}\frac{\Ewp[\meanGamma{S_M^k(X,\Gamma)}]}{\sigma^k}=\sum_{\ab\vdash k\atop{\ab\text{ has no singleton blocks}}}c(\ab)\lambda^{l(\ab)-k/2}.
$$
\cite[(3.5)]{Pri11} proves that the right hand side is $\mathbb{E}\left[
\parenth{
\frac{P-\lambda}{\sqrt\lambda}
}^k
\right]$, hence the result.
\end{proof}

\section{The asymptotics of $G_M(X,\Gamma)$ for a typical index $q$ lattice}

\subsection{An analog of the Barban-Davenport-Halberstam theorem}
Write $G_M(X,\Gamma)=qS_M^2(X,\Gamma)$. This quantity is closely related to a conjecture of Hooley (see Section \ref{sec_number_theoretic_discussoin}) and to the work of Rudnick \cite{Rud26b} (where he also considers non-primitive geodesics). Define
$$
H_M(X,q)=\meanGamma{G_M(X,\Gamma)}.
$$
Estimating $H_M(X,q)$ is somewhat analogous to the setting of the Barban-Davenport-Halberstam theorem; see Section \ref{sec_thms_bdh} below. Our results imply the following
\begin{cor}\label{cor_bdh_type_thm}
Let $q=q(X)$ be a sequence of odd primes (which are allowed to be either all constant or tend to $\infty$ with $X$). We have, as $X\to\infty$,
$$
\lim_{g\to\infty}\Ewp[H_M(X,q)]\sim\parenth{1-\frac{1}{q}}X\log X.
$$
\end{cor}

\begin{proof}
From Proposition \ref{prop_S_M_mean_as_sum_over_C_alphas} it follows that
$$
\lim_{g\to\infty}\Ewp[H_M(X,q)]=\frac{1}{q}\lim_{g\to\infty}\sum_{\alpha\in\F_q^\times}\Ewp[C_{(\alpha)}^2(X)]. 
$$
Let $\pi\vdash \{1,2\}$ be the only set partition without singleton blocks. By Proposition \ref{prop_sum_c_alphas_pi_in_terms_of_integrals_and_S_pi} and Corollary \ref{cor_partitions_with_block_of_1_do_not_contribute}, since $\mathcal{S}(\pi)=2(q-1)$, we get that 
\begin{equation}\label{eqn_H_M_asymptotics}
\lim_{g\to\infty}\Ewp[H_M(X,q)]=2\parenth{1-\frac{1}{q}}I_2(X).    
\end{equation}
The result now follows from Proposition \ref{prop_J_d_principal_bound}.
\end{proof}

\subsection{The variance of $G_M(X,\Gamma)$ over index $q$ lattices}
We wish to understand also the asymptotic behavior of $\lim_{g\to\infty}\Ewp[G_M(X,\Gamma)]$ for a typical lattice $\Gamma$ of index $q$. For this purpose, we want to compute the variance of $G_M(X,\Gamma)$ over index $q$ lattices, namely
$$
\lim_{g\to\infty}\left[
\meanGamma{G_M(X,\Gamma)^2}-\meanGamma{G_M(X,\Gamma)}^2
\right].
$$
When analyzing $G_M(X,q\Z^{2g})$ one immediately encounters the quantity 
$$
J_M(X,q)=\sum_{N(\gamma_1),N(\gamma_2)\le X\atop{[\gamma_1]=[\gamma_2]\pmod{q}}}\Lambda(\gamma_1)\Lambda(\gamma_2),
$$
defined in Rudnick \cite{Rud26b} (though as we remarked previously, in his sum he considers also non-primitive geodesics). In our notation, this is $C_{(1)}^2(X)$.

\begin{prop}\label{prop_mean_GM_as_J_M}
 We have
    $$
    \lim_{g\to\infty}\Ewp\left[
    \meanGamma{G_M(X,\Gamma)}^2
    \right]=\parenth{1-\frac{1}{q}}^2\lim_{g\to\infty}\Ewp[J_M(X,q)^2].
    $$
\end{prop}

\begin{proof}
By Proposition \ref{prop_S_M_mean_as_sum_over_C_alphas}, since $C_{(\alpha)}^2(X)=C_{(\beta)}^2(X)=J_M(X,q)$ for every $\alpha,\beta\in\F_q^\times$, 
$$
\meanGamma{G_M(X,\Gamma)}^2=\parenth{1-\frac{1}{q}}^2J_M(X,q)^2+\Oh{\frac{\Theta(X)^{4}}{q^{2g}}}.
$$
Taking $\Ewp$ and letting $g\to\infty$ we get the result.
\end{proof}

\begin{prop}\label{prop_variance_gamma_first_form}
We have
\begin{multline*}
\lim_{g\to\infty}\Ewp\left[
\meanGamma{G_M(X,\Gamma)^2}-\meanGamma{G_M(X,\Gamma)}^2
\right]=\\
=\parenth{\frac{1}{q}-\frac{1}{q^2}}\lim_{g\to\infty}\Ewp\left[
\sum_{l(\gamma_1),\ldots,l(\gamma_4)\le \log X\atop{\dim\mathrm{span}_{\F_q}([\gamma_1]-[\gamma_2],[\gamma_3]-[\gamma_4])=1\atop{[\gamma_1]-[\gamma_2],[\gamma_3]-[\gamma_4]\neq 0\pmod{q}}}}\Lambda(\gamma_1)\Lambda(\gamma_2)\Lambda(\gamma_3)\Lambda(\gamma_4)
\right].
\end{multline*}
\end{prop}

\begin{proof}
By Proposition \ref{prop_recursion_for_S_M_k}, we have
$$
G_M(X,\Gamma)=\sum_{l(\gamma_1),l(\gamma_2)\le \log X}\Lambda(\gamma_1)\Lambda(\gamma_2)1_{[\gamma_1]=[\gamma_2]\pmod{\Gamma}}-\frac{\Theta(X)^2}{q}.
$$
Squaring this we get 
\begin{multline*}
G_M(X,\Gamma)^2=\sum_{l(\gamma_1),\ldots,l(\gamma_4)\le \log X}\Lambda(\gamma_1)\Lambda(\gamma_2)\Lambda(\gamma_3)\Lambda(\gamma_4)1_{[\gamma_1]=[\gamma_2]\pmod{\Gamma}}1_{[\gamma_3]=[\gamma_4]\pmod{\Gamma}}-\\
-2\frac{\Theta(X)^2}{q}\sum_{l(\gamma_1),l(\gamma_2)\le \log X}\Lambda(\gamma_1)\Lambda(\gamma_2)1_{[\gamma_1]=[\gamma_2]\pmod{\Gamma}}+\frac{\Theta(X)^4}{q^2}.
\end{multline*}
Taking the average over all $\Gamma\in L_{g,q}$, by a similar proof to that of Proposition \ref{prop_k_vectors_on_lattice}, we see that
$$
\meanGamma{1_{[\gamma_1]=[\gamma_2]\pmod{\Gamma}}1_{[\gamma_3]=[\gamma_4]\pmod{\Gamma}}}=\begin{cases}
    1, & [\gamma_1]-[\gamma_2]=[\gamma_3]-[\gamma_4]=0\pmod{q}, \\
    q^{-1}+O(q^{-2g}), & \dim\mathrm{span}_{\F_q}([\gamma_1]-[\gamma_2],[\gamma_3]-[\gamma_4])=1,\\
    q^{-2}+\Oh{q^{-2g}}, & \text{otherwise.}
\end{cases}
$$
We get that

\begin{multline*}
\meanGamma{G_M(X,\Gamma)^2}=\sum_{l(\gamma_1),\ldots,l(\gamma_4)\le \log X\atop{[\gamma_1]=[\gamma_2]\pmod{q}\atop{[\gamma_3]=[\gamma_4]\pmod{q}}}}\Lambda(\gamma_1)\Lambda(\gamma_2)\Lambda(\gamma_3)\Lambda(\gamma_4)+\\
+\frac{1}{q}\sum_{l(\gamma_1),\ldots,l(\gamma_4)\le \log X\atop{\dim\mathrm{span}_{\F_q}([\gamma_1]-[\gamma_2],[\gamma_3]-[\gamma_4])=1}}\Lambda(\gamma_1)\Lambda(\gamma_2)\Lambda(\gamma_3)\Lambda(\gamma_4)
+\\
+\frac{1}{q^2}\sum_{l(\gamma_1),\ldots,l(\gamma_4)\le \log X\atop{\dim\mathrm{span}_{\F_q}([\gamma_1]-[\gamma_2],[\gamma_3]-[\gamma_4])=2}}\Lambda(\gamma_1)\Lambda(\gamma_2)\Lambda(\gamma_3)\Lambda(\gamma_4)
-\\
-2\frac{\Theta(X)^2}{q}\sum_{l(\gamma_1),l(\gamma_2)\le \log X\atop{[\gamma_1]=[\gamma_2]\pmod{q}}}
-2\frac{\Theta(X)^2}{q^2}\sum_{l(\gamma_1),l(\gamma_2)\le \log X\atop{[\gamma_1]\neq[\gamma_2]\pmod{q}}}+\frac{\Theta(X)^4}{q^2}+
\Oh{\frac{\Theta(X)^4}{q^{2g}}}.
\end{multline*}
Rearranging we get 
\begin{multline}\label{eqn_midterm_calc_G_M_X_sqrd}
\meanGamma{G_M(X,\Gamma)^2}=\parenth{1-\frac{1}{q^2}}J_M(X,q)^2+\\
+\parenth{\frac{1}{q}-\frac{1}{q^2}}\sum_{l(\gamma_1),\ldots,l(\gamma_4)\le \log X\atop{\dim\mathrm{span}_{\F_q}([\gamma_1]-[\gamma_2],[\gamma_3]-[\gamma_4])=1}}\Lambda(\gamma_1)\Lambda(\gamma_2)\Lambda(\gamma_3)\Lambda(\gamma_4)-\\
-2\parenth{1-\frac{1}{q}}\frac{\Theta(X)^2}{q}J_M(X,q)+\Oh{\frac{\Theta(X)^4}{q^{2g}}}.
\end{multline}
Notice that by inclusion-exclusion,
\begin{multline*}
\sum_{l(\gamma_1),\ldots,l(\gamma_4)\le \log X\atop{\dim\mathrm{span}_{\F_q}([\gamma_1]-[\gamma_2],[\gamma_3]-[\gamma_4])=1}}\Lambda(\gamma_1)\Lambda(\gamma_2)\Lambda(\gamma_3)\Lambda(\gamma_4)=\\
=\sum_{l(\gamma_1),\ldots,l(\gamma_4)\le \log X\atop{\dim\mathrm{span}_{\F_q}([\gamma_1]-[\gamma_2],[\gamma_3]-[\gamma_4])=1\atop{[\gamma_1]-[\gamma_2],[\gamma_3]-[\gamma_4]\neq 0\pmod{q}}}}\Lambda(\gamma_1)\Lambda(\gamma_2)\Lambda(\gamma_3)\Lambda(\gamma_4)+2J_M(X,q)(\Theta(X)^2-J_M(X,q)). 
\end{multline*}
Plugging this back into \eqref{eqn_midterm_calc_G_M_X_sqrd} we get that
\begin{multline*}
\meanGamma{G_M(X,\Gamma)^2}=\parenth{1-\frac{1}{q}}^2J_M(X,q)^2+\\
+\parenth{\frac{1}{q}-\frac{1}{q^2}}\sum_{l(\gamma_1),\ldots,l(\gamma_4)\le \log X\atop{\dim\mathrm{span}_{\F_q}([\gamma_1]-[\gamma_2],[\gamma_3]-[\gamma_4])=1\atop{[\gamma_1]-[\gamma_2],[\gamma_3]-[\gamma_4]\neq 0\pmod{q}}}}\Lambda(\gamma_1)\Lambda(\gamma_2)\Lambda(\gamma_3)\Lambda(\gamma_4)+\Oh{\frac{\Theta(X)^4}{q^{2g}}}.
\end{multline*}
Now using Proposition \ref{prop_mean_GM_as_J_M} we finish.
\end{proof}

\begin{definition}
Let $X\ge 0$, $q$ prime. Define for $\alpha\in \F_q^\times$
$$
S_{\alpha}(X,q)=\sum_{N(\gamma_1),\ldots,N(\gamma_4)\le X\atop {[\gamma_1]-[\gamma_2]=\alpha([\gamma_3]-[\gamma_4])\pmod {q}\atop{[\gamma_1]-[\gamma_2]\neq 0\pmod{q}}}}\Lambda(\gamma_1)\Lambda(\gamma_2)\Lambda(\gamma_3)\Lambda(\gamma_4).
$$
\end{definition}

\begin{cor}\label{cor_variance_over_gamma_explicit}
We have 
$$
\lim_{g\to\infty}\Ewp[\meanGamma{G_M(X,\Gamma)^2}-\meanGamma{G_M(X,\Gamma)}^2]=\parenth{\frac{1}{q}-\frac{1}{q^2}}\sum_{\alpha\in\F_q^\times}\lim_{g\to\infty}\Ewp[S_\alpha(X,q)].
$$
\end{cor}

\begin{proof}
This follows directly from Proposition \ref{prop_variance_gamma_first_form} and the definition of $S_\alpha(X,q)$.
\end{proof}

\subsection{Estimating the number of solutions to $[\gamma_1]-[\gamma_2]=\alpha([\gamma_3]-[\gamma_4])$} For the computation of the asymptotic behavior of $\lim_{g\to\infty}\Ewp[\meanGamma{G_M(X,\Gamma)^2}$ we will need to understand the expected value of the sums $S_\alpha(X,q)$.

\begin{prop}\label{prop_S_alpha_estimate}
We have, as $X\to\infty$,
$$\lim_{g\to\infty}\Ewp \left[
\sum_{\alpha\in\F_q^\times}S_\alpha(X,q)
\right]=4X^2\log(X)^2+\Oh{X^2\log X}.
$$
\end{prop}

\begin{proof}
Write 
$$
S_\alpha(X,q)=\sum_{N(\gamma_1),\ldots,N(\gamma_4)\le X\atop {[\gamma_1]-[\gamma_2]=\alpha([\gamma_3]-[\gamma_4])\pmod {q}\atop{[\gamma_1]-[\gamma_2]\neq 0\pmod{q}}}}\Lambda(\gamma_1)\Lambda(\gamma_2)\Lambda(\gamma_3)\Lambda(\gamma_4).
$$
By Lemma \ref{lem_non_SNS_are_negligible}, after taking $\Ewp$ and letting $g\to\infty$, we remain only with the contribution of SNS tuples. We consider their contribution based on the different possible collision types.

\case{Collision of type $\ab$, $\ab$ has a part of size $1$} Let $\ab\vdash 4$ be a partition with at least one part equal to $1$. Suppose that the tuple $\gamma_1,\ldots,\gamma_4$ has a collision of type $\ab$. Without loss of generality, assume that $\gamma_1$ does not collide with any other geodesic. Write
$$
\gamma_i=\eta_i^{\epsilon_i},\epsilon_i=\pm 1, \eta_1\neq \eta_j\text{ for }j\neq 1.
$$
For every given $\alpha\in\F_q^\times$, we have the congruence 
$$
\epsilon_1[\eta_1]=\epsilon_2[\eta_2]+\alpha\epsilon_3[\eta_3]-\alpha\epsilon_4[\eta_4]\pmod{q}.
$$
Here, $\eta_1$ is different from $\eta_2,\eta_3,\eta_4$. The latter might collide, but the underlying tuple of distinct primitive geodesics must be SNS. Therefore, we must have $\epsilon_1=0\pmod{q}$, which is a contradiction. Hence we get no contribution at all from this collision type when we take $\Ewp$ and let $g\to\infty$.

\case{Collision of type $(2,2)$} Here we have two different options: cross collision, meaning $\gamma_1$ collides with either $\gamma_3$ or $\gamma_4$, or no cross collision, when $\gamma_1$ collides with $\gamma_2$. We begin by showing that the case of no cross collision gives no contribution. And indeed, this means that 
$$
\gamma_1=\eta_1^{\epsilon_1},\gamma_2=\eta_1^{\epsilon_2},\gamma_3=\eta_3^{\epsilon_3},\gamma_4=\eta_3^{\epsilon_4},\epsilon_i=\pm 1,
$$
with $\eta_1,\eta_3$ a SNS pair. Then, for every fixed $\alpha$, the homology condition implies 
$$
(\epsilon_1-\epsilon_2)[\eta_1]=\alpha(\epsilon_3-\epsilon_4)[\eta_3]\pmod{q},
$$
which in turn implies $\epsilon_1=\epsilon_2$, so that $[\gamma_1]=[\gamma_2]\pmod{q}$ which we assume does not happen. 

Therefore, we remain with the cross collision case. Without loss of generality we assume $\gamma_1$ collides with $\gamma_3$ (the other case gives the same contribution). Then, we have
$$
\gamma_1=\eta_1^{\epsilon_1},\gamma_2=\eta_2^{\epsilon_2},\gamma_3=\eta_1^{\epsilon_3},\gamma_4=\eta_2^{\epsilon_4},\epsilon_i=\pm 1.
$$
For every given $\alpha$, the homology condition translates to 
$$
\epsilon_1=\alpha\epsilon_3\pmod{q},\epsilon_2=\alpha\epsilon_4\pmod{q}.
$$
When $\alpha\neq\pm 1$, we have no solution with $\epsilon_i=\pm 1$, so the case of cross collision also gives no contribution. When $\alpha=\pm 1$, there are precisely $4$ possible solutions. Then, we get that the contribution of the cross collision case is 
$
4\sum_{l(\eta_1),l(\eta_2)\le \log X}l(\eta_1)^2l(\eta_2)^2
$. Using propositions \ref{prop_collision_integrals} and \ref{prop_J_d_principal_bound}, we get that
$$
4\lim_{g\to\infty}\Ewp\left[
\sum_{l(\eta_1),l(\eta_2)\le \log X}l(\eta_1)^2l(\eta_2)^2
\right]=4I_2(X)^2=X^2\log(X)^2+\Oh{X^2\log X}.
$$
Since there are $2$ possible set partitions of $\{1,2,3,4\}$ of type $(2,2)$ with cross collision, we get that the total contribution of this case is indeed 
$$
1_{\alpha=\pm 1}\parenth{ 2X^2\log(X)^2+\Oh{X^2\log X}}.
$$

\case{Collision of type $(4)$} Since we assume
$$
[\gamma_1]-[\gamma_2]\neq 0\pmod{q},
$$
there can be solutions to
$$
[\gamma_1]-[\gamma_2]=\alpha([\gamma_3]-[\gamma_4])\pmod{q}
$$
if and only if $\alpha=\pm 1$. Therefore we bound the contribution of this case by 
$
\ll \sum_{l(\eta_1)\le \log X}l(\eta_1)^4
$. Using propositions \ref{prop_collision_integrals} and \ref{prop_J_d_principal_bound}, we get that
$$
\lim_{g\to\infty}\Ewp\left[
\sum_{l(\eta_1)\le \log X}l(\eta_1)^4
\right]=I_4(X)\ll X\log(X)^3,
$$
which is absorbed into the error term.
\end{proof}

\begin{cor}\label{cor_variance_over_lattices}
We have, as $X\to\infty$,
\begin{equation*}
\lim_{g\to\infty}\Ewp[\meanGamma{G_M(X,\Gamma)^2}-\meanGamma{G_M(X,\Gamma)}^2]
=\frac{4q-4}{q^2}X^2\log^2 X+\Oh{\frac{1}{q}X^2\log X}.
\end{equation*}
\end{cor}

\begin{proof}
This follows directly from Proposition \ref{prop_S_alpha_estimate} and Corollary \ref{cor_variance_over_gamma_explicit}.
\end{proof}

\subsection{The second moment of $J_M(X,q)$} In order to show that $\meanGamma{G_M(X,\Gamma)}$ concentrates around $X\log X$, we will need to compute the second moment of $J_M(X,q)$.

\begin{prop}\label{prop_expected_val_J_M}
As $X\to\infty$, we have
$$
\lim_{g\to\infty}\Ewp[J_M(X,q)^2]=X^2\log(X)^2+\Oh{X^2\log X}.
$$
\end{prop}

\begin{proof}
We have
$$
J_M(X,q)^2=\sum_{N(\gamma_1),\ldots,N(\gamma_4)\le X\atop{[\gamma_1]=[\gamma_2]\pmod{q}\atop{[\gamma_3]=[\gamma_4]\pmod{q}}}}\Lambda(\gamma_1)\Lambda(\gamma_2)\Lambda(\gamma_3)\Lambda(\gamma_4).
$$
We split the contribution to the different possible collision types. By Lemma \ref{lem_non_SNS_are_negligible}, we only consider the contribution of SNS tuples.

\case{Collision of type $\ab$, $\ab$ has a part of size $1$} Let $\ab\vdash 4$ be a partition with at least one part equal to $1$. Suppose that the tuple $\gamma_1,\ldots,\gamma_4$ has a collision of type $\ab$. Without loss of generality, assume that $\gamma_1$ does not collide with any other geodesic. Write
$$
\gamma_i=\eta_i^{\epsilon_i},\epsilon_i=\pm 1, \eta_1\neq \eta_j\text{ for }j\neq 1.
$$
In particular, $\eta_1,\eta_2$ form an SNS pair; we can thus assume without loss of generality that $[\eta_1]=e_1,[\eta_2]=e_2$ in $H_1(M,\Z)$. Then, the first homology condition translates to
$$
\epsilon_1e_1=\epsilon_2e_2\pmod{q},\epsilon_i=\pm 1,
$$
which is a contradiction. Therefore we get no contribution at all from this collision type when we take $\Ewp$ and let $g\to\infty$.

\case{Collision type $(4)$} In this case we have $\gamma_i=\eta_1^{\epsilon_i},\epsilon_i=\pm 1$, and $\eta_1$ a simple closed geodesic. The homology condition then forces $\epsilon_1=\epsilon_2,\epsilon_3=\epsilon_4$ and therefore the contribution of this case is $
4\sum_{N(\eta_1)\le X}l(\eta_1)^4
$. Using Proposition \ref{prop_collision_integrals} we get
$$
\lim_{g\to\infty}\Ewp\left[
\sum_{N(\eta_1)\le X}l(\eta_1)^4
\right]=I_4(X).
$$
Therefore, by Proposition \ref{prop_J_d_principal_bound}, this case contributes $\ll X\log(X)^3$.

\case{Collision type $(2,2)$}
There are two types of collisions, which contribute differently: cross collisions (when $\gamma_1$ collides with either $\gamma_3$ or $ \gamma_4$) and non-cross collisions (when $\gamma_1$ collides with $\gamma_2$). By essentially the same argument as in the first case, cross collisions give zero contribution after we take $\Ewp$ and let $g\to\infty$. Therefore we focus on the non-cross collision case. In this case we have
$$
\gamma_1=\eta_1^{\epsilon_1},\gamma_2=\eta_1^{\epsilon_2},\gamma_3=\eta_3^{\epsilon_3},\gamma_4=\eta_3^{\epsilon_4},\epsilon_i=\pm 1.
$$
The homology conditions translate then to demanding $\epsilon_1=\epsilon_2,\epsilon_3=\epsilon_4$. We get that the contribution of this case is $4\sum_{N(\eta_1),N(\eta_2)\le X}l(\eta_1)^2l(\eta_2)^2$. Using Proposition \ref{prop_collision_integrals} we get
$$
\lim_{g\to\infty}\Ewp\left[
\sum_{N(\eta_1),N(\eta_2)\le X}l(\eta_1)^2l(\eta_2)^2
\right]=I_2(X)^2.
$$
Using Proposition \ref{prop_J_d_principal_bound} we get that the total contribution of this case is $X^2\log(X)^2+\Oh{X^2\log X}$.
\end{proof}

\begin{cor}\label{cor_Var_g_G_M_X}
Define 
$$
\operatorname{Var}_g(\meanGamma{G_M(X,\Gamma)}):=\Ewp\left(
\meanGamma{G_M(X,\Gamma)}^2
\right)-\Ewp\left(
\meanGamma{G_M(X,\Gamma)}
\right)^2.
$$
We have, as $X\to\infty$, that
$$
\lim_{g\to\infty}\operatorname{Var}_g(\meanGamma{G_M(X,\Gamma)}=\Oh{X^2\log X}.
$$
\end{cor}

\begin{proof}
By Proposition \ref{prop_mean_GM_as_J_M} and equation \eqref{eqn_H_M_asymptotics} we get that
$$
\lim_{g\to\infty}\operatorname{Var}_g(\meanGamma{G_M(X,\Gamma)})=\parenth{1-\frac{1}{q}}^2 \lim_{g\to\infty}\Ewp[J_M(X,q)^2]-\parenth{1-\frac{1}{q}}^2 X^2\log^2(X)+O(X^2\log X).
$$
Now the claim follows from Proposition \ref{prop_expected_val_J_M}.
\end{proof}

\subsection{$G_M(X,\Gamma)$ for a typical index $q$ lattice} 

\begin{prop}\label{prop_average_is_concentrated}
Let $q=q(X)\to_{X\to\infty}\infty$ be a sequence of primes. We have
$$
\lim_{X\to\infty}\lim_{g\to\infty} \Ewp\left[
1_{\left|\meanGamma{G_M(X,\Gamma)}-X\log X\right|\ge \epsilon X\log X}\right]=0.
$$
\end{prop}

\begin{proof}
We have the inequality
$$
1_{\left|\meanGamma{G_M(X,\Gamma)}-X\log X\right|\ge \epsilon X\log X}\le 1_{\left|\meanGamma{G_M(X,\Gamma)}-(1-1/q)X\log X\right|\ge (\epsilon-1/q) X\log X}.
$$
We take $X$ so large that $\epsilon-1/q>\epsilon/2$. Now we bound
$$
1_{\left|\meanGamma{G_M(X,\Gamma)}-(1-1/q)X\log X\right|\ge \frac{1}{2}\epsilon X\log X}\le 
\frac{\parenth{\meanGamma{G_M(X,\Gamma)}-(1-1/q)X\log X}^2}{\frac{1}{4}\epsilon^2 X^2\log^2 X}.
$$
We get that
\begin{multline*}
\lim_{g\to\infty}\Ewp\left[
1_{\left|\meanGamma{G_M(X,\Gamma)}-X\log X\right|\ge \epsilon X\log X}\right]\ll\\
\ll\lim_{g\to\infty} \Ewp\left[
\frac{\parenth{\meanGamma{G_M(X,\Gamma)}-(1-1/q)X\log X}^2}{\frac{1}{4}\epsilon^2 X^2\log^2 X}
\right]=o(1),
\end{multline*}
using Corollary \ref{cor_Var_g_G_M_X}, since
\begin{multline*}
\Ewp\left[\parenth{
\meanGamma{G_M(X,\Gamma)}-(1-1/q)X\log X}^2
\right]=\\
=\operatorname{Var}_g(\meanGamma{G_M(X,\Gamma)})+(\Ewp[\meanGamma{G_M(X,\Gamma)}]-(1-1/q)X\log X)^2.
\end{multline*}
\end{proof}

We are now ready to prove Theorem \ref{thm_G_M_X_gamma_concentrations}.

\begin{proof}[Proof of Theorem \ref{thm_G_M_X_gamma_concentrations}] By Chebyshev's inequality, 
$$
\Pr_{\Gamma\in L_{g,q}}\left[
\left|
G_M(X,\Gamma)-\meanGamma{G_M(X,\Gamma)}
\right|>\frac{1}{2}\epsilon X\log X
\right]\le \frac{\meanGamma{G_M(X,\Gamma)^2}-\meanGamma{G_M(X,\Gamma)}^2}{\frac{1}{4}\epsilon^2 X^2\log^2 X}.
$$
Using Corollary \ref{cor_variance_over_lattices} we get that
$$
\lim_{g\to\infty}\Ewp\left[
\Pr_{\Gamma\in L_{g,q}}\left[
\left|
G_M(X,\Gamma)-\meanGamma{G_M(X,\Gamma)}
\right|>\frac{1}{2}\epsilon X\log X
\right]
\right]\ll \frac{q-1}{\epsilon^2 q^2}\ll \frac{1}{\epsilon^2 q}.
$$
But
\begin{multline*}
\Pr_{\Gamma\in L_{g,q}}\left[
|G_M(X,\Gamma)-X\log X|>\epsilon X\log X
\right]\le\\\le \Pr_{\Gamma\in L_{g,q}}\left[
\left|
G_M(X,\Gamma)-\meanGamma{G_M(X,\Gamma)}
\right|>\frac{1}{2}\epsilon X\log X
\right]+1_{\left|\meanGamma{G_M(X,\Gamma)}-X\log X\right|\ge \frac{1}{2}\epsilon X\log X}.    
\end{multline*}
We get that 
\begin{multline*}
\lim_{g\to\infty} \Ewp\left[
\frac{1}{|L_{g,q}|}\#\left\{\Gamma\in L_{g,q}: \left|\frac{G_M(X,\Gamma)}{X\log X}-1\right|>\epsilon\right\}
\right]\ll\\
\ll\frac{1}{\epsilon^2 q}+\lim_{g\to\infty} \Ewp\left[
1_{\left|\meanGamma{G_M(X,\Gamma)}-X\log X\right|\ge \frac{1}{2}\epsilon X\log X}
\right].    
\end{multline*}
Using Proposition \ref{prop_average_is_concentrated}, taking $X\to\infty$ we get the result.
\end{proof}

\section{Comparison with the distribution of primes in arithmetic progressions}\label{sec_number_theoretic_discussoin}
Define 
$$
\Theta(X;Q,A)=\sum_{p\le X\atop{p\equiv A\pmod{Q}}}\log p,
$$
where the sum is over primes $p$. The Prime Number Theorem for arithmetic progressions states that for a fixed modulus $Q$, and $A$ with $\gcd(A,Q)=1$,
$$
\Theta(X;Q,A)\sim \frac{X}{\phi(Q)},\text{ as }X\to\infty.
$$

An interesting question is: how does the quantity $\Theta(X;Q,A)$ fluctuate around its mean value (which is $\frac{X}{\phi(Q)}$)? Hooley studied this question in a long series of papers spanning from the 1970s to the 2000s. Define 
$$
E(X;Q,A)=\Theta(X;Q,A)-\frac{X}{\phi(Q)}.
$$
The quantity
$$
G(X,Q)=\sum_{A\pmod{Q}\atop{\gcd}}E(X;Q,A)^2
$$
measures the variance of primes in arithmetic progressions for a fixed modulus $Q$. We can also study the averaged quantity 
$$
H(X,Q)=\sum_{Q'\le Q}G(X,Q').
$$

We discuss the work of Hooley and others on this problem, and how it relates to our results. For a longer survey on the theory of primes in arithmetic progressions, we recommend Hooley's ICM survey \cite{Hoo74}.

\subsection{Hooley's conjecture}

Hooley \cite{Hoo74} conjectured that under some unspecified conditions, for $X>Q$,
\begin{equation}\label{eqn_hooley_nontrivial_regime} 
G(X,Q)\sim X\log Q.
\end{equation}
In the "trivial" regime $Q^{\epsilon}<X<Q$, Friedlander and Goldston \cite{FriGol96} proved that 
$$
G(X,Q)\sim X\log X.
$$
Also, assuming a Hardy-Littlewood conjecture, they proved that \eqref{eqn_hooley_nontrivial_regime} holds in the range $X^{1/2+\epsilon}<Q<X$.

Rudnick and Keating \cite{RudKea14} considered a function field analog of Hooley's conjecture, and proved a result which is analogous to \eqref{eqn_hooley_nontrivial_regime} in the range $Q<X^{1-\epsilon}$. As a result, the estimate \eqref{eqn_hooley_nontrivial_regime} is believed to hold in the range $X^{\epsilon}<Q<X^{1-\epsilon}$. 

\subsubsection{Geodesics in homology classes} In \cite{Rud26b}, Rudnick considers the variance of homology classes mod $q$. Define, for $\alpha\in H_1(M,\Z/q\Z)$,
$$
\Psi_M(X;q,\alpha)=\sum_{N(\gamma)\le X\atop [\gamma]=\alpha\pmod{q}}\Lambda(\gamma), \tilde{G}_M(X,q)=\sum_{\alpha\in H_1(M,\Z/q\Z)} \parenth{\Psi_M(X;q,\alpha)-\frac{\Psi_M(X)}{q^{2g}}}^{2}.
$$
Here the sum $\Psi_M(X;q,\alpha)$ runs also over non-primitive $\gamma$. The quantity $\tilde{G}_M(X,q)$ is closely related to $G_M(X,q\Z^{2g})$ (the difference being that Rudnick also considers non-primitive geodesics). Rudnick proves that as $X\to\infty$,
\begin{equation}\label{eqn_rudnicks_result_hooley_geodesics}
\lim_{g\to\infty}\Ewp[\tilde{G}_M(X,q)]\sim\begin{cases}
    X\log X, & q>2,\\
    2X\log X, & q=2.
\end{cases}
\end{equation}

To explain the discrepancy from Hooley's conjecture \eqref{eqn_hooley_nontrivial_regime}, Rudnick \cite[\S 6]{Rud26b} introduces a random model for the distribution of closed geodesics in homology classes, based on the random allocation of weighted balls. This model can be seen to explain our results as well.

We note that since the volume of the lattice $q\Z^{2g}$ goes to infinity when $g\to\infty$, \eqref{eqn_rudnicks_result_hooley_geodesics} remains in the "trivial regime" where we expect most residue classes to contain no primitive geodesics of length $\le \log X$. Our Theorem \ref{thm_G_M_X_gamma_concentrations} shows that as $q=q(X)\to\infty$, for most lattices $\Gamma\in L_{g,q}$ of index $q$, a random Weil-Petersson $M$ has $G_M(X,\Gamma)$ asymptotically close to $X\log X$, thereby providing further evidence to Rudnick's random model.

\subsection{Theorems of Barban-Davenport-Halberstam type}\label{sec_thms_bdh}
The study of $H(X,Q)$ goes by the name of Barban-Davenport-Halberstam type theorems. Independent work of Barban and of Davenport-Halberstam gives an upper bound, which we state here in an improved form due to Gallagher \cite{Gal67}:

$$
H(X,Q)=O(QX\log X)+O(X^2\log^{-A}X),\text{ for any positive constant }A, Q\le X.
$$
Montgomery \cite{Mon70} and Hooley \cite{Hoo75a} gave an explicit asymptotic estimate. For every $A>0$, $X/(\log X)^A<Q<X$, they proved that
\begin{equation}\label{eqn_H_X_Q_estimate}
H(X,Q)=QX\log Q-cQX+\Oh{Q^{5/4}X^{3/4}+\frac{X^2}{(\log X)^A}},    
\end{equation}
where 
$$
c=\gamma+\log(2\pi )+1+\sum_p \frac{\log p}{p(p-1)}.
$$
Assuming GRH, Hooley \cite{Hoo75b} proved that \eqref{eqn_H_X_Q_estimate} holds for $X^{1/2+\epsilon}<Q<X$, for every $\epsilon>0$, with remainder $\Oh{X^2/(\log X)^A}$. 

Since our quantity $H_M(X,Q)=\meanGamma{G_M(X,\Gamma)}$ averages $G_M(X,\Gamma)$ over all sublattices of $\Z^{2g}$ of fixed index $q$, we view this situation as somewhat analogous to the study of $H(X,Q)$. Motivated by Rudnick's result \eqref{eqn_rudnicks_result_hooley_geodesics} and our Theorem \ref{thm_G_M_X_gamma_concentrations}, one may speculate that in the geodesics world, the correct main term in the asymptotic for $H_M(X,q)$ for a typical surface should be $X\log X$ (instead of the $X\log Q$ obtained by normalizing the main term in \eqref{eqn_H_X_Q_estimate} by the number of moduli being averaged). This is what our Corollary \ref{cor_bdh_type_thm} shows in the large genus limit.

\subsubsection{Higher moments} Hooley \cite{Hoo77} studies the limiting distribution of $E(X;Q,A)$ and the higher moments of it. Let $0\le t\le U$ be chosen at random, and define the random variable 
$$
\frac{E(e^t;Q,A)}{\sqrt{e^t\log Q/\phi(Q)}}.
$$
Under ERH and the LI hypothesis, Hooley proves that as $U\to\infty$, this random variable has a continuous distribution function $F(y)$ which is independent of $A$ for any given $Q$. Moreover, $\lim_{Q\to\infty}F(y)=\Phi(y)$, where $\Phi(y)$ is the distribution function of the normal distribution $\mathcal{N}(0,1)$.  

Even more related to our results, Hooley \cite{Hoo77} makes the conjecture that for $Q\to\infty$, with $\frac{X}{Q\log X}\to\infty$, $Q>X\log ^{-A}X$ for some constant $A>0$,
$$
\sum_{Q'\le Q}\phi^{\frac{k}{2}-1}(Q')\sum_{0<A\le Q'\atop{\gcd(A,Q')=1}}E(X;Q',A)^k=\begin{cases}
    o\parenth{QX^{k/2}\log(Q)^{k/2}}, & k\text{ is odd,}\\
    ((k-1)!!+o(1))QX^{k/2}\log(Q)^{\frac{k}{2}}, & k\text{ is even.}
\end{cases}
$$

In a subsequent paper \cite{Hoo98}, he then computes the third moment unconditionally. Denoting 
$$
S(Q)=\sum_{Q'\le Q}\phi(Q')\sum_{0<A\le Q'\atop{\gcd(A,Q')=1}}E(X;Q',A)^3,
$$
he proves that for any positive constant $A$ and any $Q$ with $X/Q\log X\to\infty$ we have
$$
S(Q)=o\parenth{Q^{3/2}X^{3/2}\log(X)^{3/2}}+\Oh{\frac{X^3}{\log (X)^A}}.
$$
These results suggest that in the regime $\frac{X}{Q\log X}\to\infty$, the errors $E(X;Q,A)$ have Gaussian behavior.

\subsubsection{Transition from Poisson to Gaussian} One can ask about the regime $\frac{X}{Q\log X}\to\lambda>0$, and the transition to the case where $\frac{X}{Q\log X}\to\infty$. In the context of counting primes in short intervals, namely studying either
$$
\Theta(X;H)=\sum_{X\le p\le X+H}\log p, H=o(X),
$$
or 
$$
\pi(X;H)=\#\{\text{primes } X\le p\le X+H\}, H=o(X),
$$
there is an established transition from Gaussian behavior in the regime $H/\log X\to\infty$ to Poisson behavior in the regime $H/\log X\to\lambda$. Assuming a certain version of the
prime $r$-tuple conjecture of Hardy and Littlewood, Gallagher \cite{Gal76} proved that for $H=\lambda \log X$,
$$
\frac{1}{X}\#\{n\le X:\pi(n;H)=k\}\to_{X\to\infty}\frac{e^{-\lambda}\lambda^k}{k!}.
$$
Montgomery and Soundarajan \cite{MonSound04}, assuming again a version of the conjecture of Hardy and Littlewood, prove that for the regime $\frac{H}{\log X}\to\infty$ (under the additional restriction that $\frac{\log H}{\log X}\to 0$) we get a Gaussian behavior. 

Much more recently, Leung \cite{Leung26} considered the transition in the case of distribution in arithmetic progressions, without the extra averaging over $Q$. Assuming a version of the mentioned Hardy and Littlewood conjecture, he proves that for $\frac{X}{\phi(Q)\log X}\to\lambda>0$, 
$$
\frac{1}{\phi(Q)}\sum_{0<A\le Q\atop{\gcd(A,Q)=1}}\pi(X;Q,A)^k\to_{X\to\infty} \E[P_\lambda^k],P_\lambda\sim\mathrm{Poi(\lambda)}.
$$
Here, 
$$
\pi(X;Q,A)=\#\{\text{primes }p\le X,p\equiv A\pmod{Q}\}.
$$
Using the same conjecture, in the range $\frac{X}{\phi(Q)\log X}\to\infty$ (with additional restriction on the size of $\phi(Q)$, which is roughly $\frac{\log(\phi(Q))}{\log X}\to 1$) he proves that
$$
\frac{1}{\phi(Q)}\sum_{0<A\le Q\atop{\gcd(A,Q)=1}}\parenth{\frac{1}{\sqrt{X\log Q/\phi(Q)}}E(X;Q,A)}^k\to_{X\to\infty} \begin{cases}
    0, & k \text{ is odd,}\\
    (k-1)!!, & k\text{ is even.}
\end{cases}
$$

Therefore, the transition we observe in the setting of prime geodesics in homology classes also appears for prime numbers (conditionally).

\bibliography{mybib}

@book{FarbMargalit,
  author    = {B. Farb and D. Margalit},
  title     = {A Primer on Mapping Class Groups},
  series    = {Princeton Mathematical Series},
  volume    = {49},
  publisher = {Princeton University Press},
  address   = {Princeton, NJ},
  year      = {2012}
}

@Article{FriGol96,
	author = {J.~B. Friedlander and D.~A. Goldston}, 
	title = {Variance of distribution of primes in residue classes}, 
	journal = {Quart. J. Math. Oxford Ser. (2)},
	volume = {47},
    number = {187},
	year = 1996, 
	pages = {313–336}}

@Article{Gal76,
author = {P. X. Gallagher},
title = {On the distribution of primes in short intervals},
journal = {Mathematika},
volume = {23},
number = {1},
year = 1976,
pages = {4-9}}

@Article{Gal67,
author = {P. X. Gallagher},
title = {The large sieve},
journal = {Mathematika},
volume = {14},
year = 1967,
pages = {14-20}}

@Article{Hoo74,
author = {C. Hooley},
title = {The Distribution of Sequences in Arithmetic Progressions},
journal = {Proc. ICM Vancouver},
year = 1974,
pages = {357-364}}

@Article{Hoo75a,
	author = "C. Hooley",
	title = "On the {Barban}-{Davenport}-{Halberstam} theorem {I}",
	journal = "Collection of articles dedicated to {Helmut} {Hasse} on his seventy-fifth birthday, III. J. Reine Angew. Math. 274/275",
	pages = "206-223", 
	year = 1975}

@Article{Hoo75b,
	author = "C. Hooley",
	title = "On the {Barban}-{Davenport}-{Halberstam} theorem {II}",
	journal = "J. London Math. Soc. (2)",
    volume = 9,
    number = 4,
	pages = "625-636", 
	year = 1975}

@Article{Hoo77,
	author = "C. Hooley",
	title = "On the {Barban}-{Davenport}-{Halberstam} theorem {VII}",
	journal = "J. London Math. Soc. (2)",
    volume = 16,
    number = 1,
	pages = "1-8", 
	year = 1977}

@Article{Hoo98,
	author = "C. Hooley",
	title = "On the {Barban}-{Davenport}-{Halberstam} theorem {VIII}",
	journal = "J. reine angew. Math.",
    volume = 499,
	pages = "1-46", 
	year = 1998}

@Article{Leung26,
	author = "S.~K. Leung", 
	title = "Moments of primes in progressions to a large modulus",
	journal = "Forum. Math.",
	volume = "38", 
    number = "2",
	pages = "555-574",
	year = 2026}

@Article{MP19,
	author = "M. Mirzakhani and B. Petri",
	title = "Lengths of closed geodesics on random surfaces
of large genus",
	journal = "Comment. Math. Helv",
	volume = "94",
	number = "4",
	pages = "869-889",
	year = 2019}

@Article{Mon70,
	author = "H.~L. Montgomery", 
	title = "Primes in arithmetic progressions",
	journal = "Michigan Math. J.",
	number = "17", 
	pages = "33-39",
	year = 1970}

@Article{MonSound04,
	author = "H.~L. Montgomery and K. Soundarajan", 
	title = "Primes in short intervals",
	journal = "Comm. Math. Phys.",
	volume = "252", 
    number = "1",
	pages = "589-617",
	year = 2004}

@article{PhillipsSarnak1987,
  author  = {R. Phillips and P. Sarnak},
  title   = {Geodesics in Homology Classes},
  journal = {Duke Math. J.},
  volume  = {55},
  number  = {2},
  year    = {1987},
  pages   = {287-297},
}

@Article{Pri11,
	author = "N. Privault",
	title = "Generalized {Bell} polynomials and the combinatorics of {Poisson} central moments",
	journal = "Electron. J. Comb.", 
	volume = "18",
    number = "1",
	year = 2011}

@Article{Rud26b,
	author = "Z. Rudnick",
	title = "Closed geodesics in homology classes on random hyperbolic surfaces of large genus",
	year = 2026,
	journal = "arXiv:2607.06263 [Math.GT]"}

@article{RudKea14,
	title={The variance of the number of prime polynomials in short intervals and in residue classes},
	author={J.~P. Keating and Z. Rudnick},
	journal={Int. Math. Res. Not. IMRN},
	number={1},
	pages={259-288},
	year=2014
}

@Article{Sharp04,
	author = {R. Sharp},
	title = {A local limit theorem for closed geodesics and homology},
	journal = {Trans. Amer. Math. Soc.},   
    volume  = {356},
    number  = {12},
    pages   = {4897-4908},
    year    = {2004}
}
\bibliographystyle{alpha}

\end{document}